\documentclass[11pt,reqno]{amsart}
\usepackage{graphicx}
\pdfoutput=1
\usepackage{amsfonts,amsmath,amssymb}
\usepackage{hyperref}
\usepackage{paralist}
\usepackage[usenames]{xcolor}
\renewcommand{\Re}{\tmop{Re}}
\renewcommand{\Im}{\tmop{Im}}

\renewcommand {\a}{\alpha}
\renewcommand {\b}{\beta}

\usepackage{thmtools}
\declaretheoremstyle[bodyfont=\normalfont]{noncursive}
\declaretheorem{theorem}
\declaretheorem[numberwithin=section]{lemma}

\declaretheorem[numberlike=lemma]{proposition}

\declaretheorem[style=noncursive,numberlike=lemma]{definition}

\declaretheorem[style=noncursive,numberlike=lemma]{remark}

\declaretheorem[style=noncursive,numberlike=lemma]{convention}
\renewcommand{\Re}{\mathop{\rm Re}\nolimits}
\renewcommand{\Im}{\mathop{\rm Im}\nolimits}
\newcommand{\im}{\ensuremath{\mbox{\rm Im}\,}}
\newcommand{\re}{\ensuremath{\mbox{\rm Re}\,}}

\newcommand{\zz}{\ensuremath{\zeta}}

\def\1#1{\overline{#1}}
\def\2#1{\widetilde{#1}}
\def\3#1{\widehat{#1}}
\def\4#1{\mathbb{#1}}
\def\5#1{\frak{#1}}
\def\6#1{{\mathcal{#1}}}

\newcommand{\CC}[1]{\mathbb{C}^{#1}}

\newcommand{\RR}[1]{\mathbb{R}^{#1}}
\newcommand{\dw}{\frac{\partial}{\partial w}}

\newcommand{\dz}{\frac{\partial}{\partial z}}

\newcommand{\z}{\zeta}

\newcommand{\lr}{\longrightarrow}

\newcommand{\g}{\mathfrak{g}}

\numberwithin{equation}{section}

\def\Label#1{\label{#1}}

\title[A complete normal form for everywhere Levi--degenerate hypersurfaces]{A complete normal form for everywhere Levi--degenerate hypersurfaces in $\CC{3}$}

\author {M. Kol\'a\v r}
\address{Department of Mathematics and Statistics, Masaryk University, Brno}
\email{mkolar@math.muni.cz}

\author {I. Kossovskiy}\thanks{The research of M. Kol\'a\v r and I.~Kossovskiy was supported by the Czech Grant Agency (GACR) grants GA17-19437S and GA21-09220S. I.\,Kossovskiy was also supported by the Austrian Science Fund (FWF) grant P29468 and P34368}
\address{\parbox{0.8\linewidth}{%
        Department of Mathematics and Statistics, Masaryk University, Brno/\\ %
        Institute of Discretet  Mathematics and Geometry, Technical University of Vienna}
    }

\email{kossovskiyi@math.muni.cz, kossovi3@univie.ac.at}


\begin{document}
\date{\today}
\maketitle

\begin{abstract}
$2$-nondegenerate real hypersurfaces in complex manifolds play an important role in CR-geometry and the theory of Hermitian Symmetric Domains. In this paper, we obtain a complete convergent normal form for everywhere $2$-nondegenerate real-analytic hypersurfaces in complex $3$-space. We do so by entirely reproducing the Chern-Moser theory in the $2$-nondegenerate setting. This seems to be the first such construction for hypersurfaces of {\em infinite Catlin multitype}. We in particular discover {\em chains} in an everywhere $2$-nondegenerate hypersurface, the tangent lines to which at a point form the so-called {\em canonical cone}. Our approach is based on using a {\em rational (nonpolynomial) model} for everywhere $2$-nondegenerate hypersurfaces,
 which is the local realization due to Fels-Kaup of the well known {\em tube over the light cone}. For the convergence of the normal form, we use an argument due to Zaitsev, based on building 
 %the so-called {\em Zaitsev's 
 a canonical direction field in an appropriate bundle over a hypersurface. 

As an application, we obtain, in the spirit of Chern-Moser theory, a criterion for the {\em local sphericity} (i.e. local equivalence to the model) for a $2$-nondegenerate hypersurface in terms of its normal form. 
As another  application, we obtain an explicit description of  the moduli space of everywhere $2$-nondegenerate hypersurfaces. 
%For the latter problem, we show that it reduces to a certain ``mixed''\, Cauchy problem with the Cauchy data on a ``cross''.  
\end{abstract}
\date{\today}
\tableofcontents

\section{Introduction}

\subsection{Historic outline}

Understanding invariants and symmetries of real hypersurfaces in complex space is one of the central goals in Several Complex Variables. The study of them was initiated
in the seminal 1907 work of Poincar\'e \cite{poincare} who discovered the finite-dimensionality of the local automorphism group for nondegenerate real hypersurfaces,
and further, the non-triviality of the holomorphic mapping problem (which is due to existence of {\em local holomorphic invariants} of hypersurfaces). 

Under the assumption of the {\em Levi-nondegeneracy} of a real hypersurface, a complete picture of invariants and symmetries has been provided in the well known works 
of Cartan \cite{cartan}, Tanaka \cite{tanaka}, Chern and Moser \cite{chern}. The differential-geometric constructions of Cartan, Tanaka, and Chern are certain developments 
of Cartan's moving frame method. They provide a solution to the problem via presenting a canonical frame on a certain bundle associated with a hypersurface.
In contrast, Moser uses an approach inspired by Dynamical Systems and provides a {\em convergent normal form} for a hypersurface. Such a normal form 
provides a distinguished choice of local holomorphic coordinates for a hypersurface, in which its defining equation is approximated "as close as possible"\, 
by  that for the {\em local quadratic model}: a real hyperquadric   
\begin{equation}\Label{quadric}
\im w=Q(z,\bar z),\quad (z,w)\in\CC{n}\times\CC{}
\end{equation}
(where $Q(z,\bar z)$ is a nondegenerate Hermitian form on $\CC{n}$). A biholomorphic transformation bringing a hypersurface to a normal form at a point is  defined uniquely, up to the automorphism group of the model \eqref{quadric}. In this sense, Moser's normal form is {\em complete}. Symmetries of a hypersurface can be read subsequently from the constructed normal form (see e.g. Beloshapka \cite{belold} and Kruzhilin-Loboda \cite{krlo}).   For recent developments in the Chern-Moser theory in high codimension see the work \cite{ls} of Lamel-Stolovitch.  

The problem of describing invariants and symmetries of a real hypersurface appears to be much more difficult
when a hypersurface is {\em Levi-degenerate} at the reference point $p\in M$. A powerful approach here is coming from the
notion of {\em Catlin multitype} of a real hypersurface \cite{catlin} at a point, which is often used in the theory of subelliptic estimates.
For hypersurfaces of finite Catlin multitype at a point, an approach of studying  symmetries and normal forms (based on using higher degree polynomial models) was developed by Kol\'a\v r-Meylan-Zaitsev \cite{kmz}
(see also the subsequent work of Kol\'a\v r-Meylan \cite{km}). It is also notable that a hypersurface of finite Catlin multitype is "good"\, in 
that it still contains Levi-nondegenerate points. More precisely, a generic point in a hypersurface is Levi-nondegenerate,
and this allows to read a lot of information on the CR-structure in a neighborhood of the degeneracy point from the nearby nondegenerate points. 

Probably one of the most intriguing phenomena in CR-geometry is the existence of {\em everywhere Levi-degenerate} real hypersurfaces in complex space of dimension $3$ and higher, which still posses a {\em finite-dimensional} local automorphism group. The relevant nondegeneracy concepts for such hypersurfaces are due to Freeman \cite{freeman} ({\em Freeman sequences}), Stanton and Baoeundi-Ebenfelt-Rothschild \cite{stanton2,beralg} ({\em holomorphic nondegeneracy}), and Baouendi-Huang-Rothschild \cite{bhr} ({\em finite nondegeneracy}). We refer to the book of Baouendi-Ebenfelt-Rothschild for 
%precise definitions and 
further details. For the hypersurfaces under discussion, the situation changes dramatically compared to the finite multitype case: since there are no Levi-nondegenerate points in $M$ {\em at all}, Cartan-Tanaka-Chern-Moser invariants at nearby points  {\em can not} be used anymore for describing invariants at a given point $p\in M$ of Levi-degeneracy.  Besides of that, everywhere $k$-nondegenerate hypersurfaces with $k\geq 2$  have {\em infinite} Catlin mutitype.

In this paper, we address an important class of infinite Catlin multitype hypersurfaces which is the class of {\em everywhere $2$-nondegenerate hypersurfaces}. The latter can be  considered, in a certain sense, as the first possible  nondegeneracy class within the category of infinite multitype hypersurfaces.
For a hypersurface $M$ sitting in $\CC{N},\,N\geq 3$, the everywhere $2$-nondegeneracy means the validity of the following two conditions. First of all, the Levi form of $M$ has everywhere 
rank $N-2$ (and hence the distribution $K$ of pointwise Levi kernels $K_p\subset T^{\CC{}}_pM$ is well defined). 
Once the first condition is satisfied, the second one can be formulated in the language of local sections of the complex tangent bundle, namely, it requires:
\begin{equation}\Label{2nondeg}
[K,T^{\CC{}}M]\not\subset K.
\end{equation}
For an alternative definition of $2$-nondegeneracy we refer to \cite{bhr} and \cite{ber}. 
We note that hypersuraces satisfying (on an open subset) the first condition but not the second are {\em holomorphically degenerate} and hence their automorphism groups are infinite-dimensional (see \cite{ber}).         

The simplest example of a $2$-nondegenerate hypersurface  is the well known {\em tube over the light cone} 
\begin{equation}\Label{tube}
C=\left\{(z_1,z_2,z_3)\in\CC{3}:\,\,y_1^2+y_2^2=y_3^2,\quad y_j=\im z_j, \,\, y_3 \neq 0\right\}\subset\CC{3}.
\end{equation}
This hypersurface has everywhere Levi rank 1 (and hence a $1$-dimensional Levi kernel) and is everywhere $2$-nondegenerate. This hypersurface has the "large"\, but still finite-dimensional automorphism group $SO(3,2)$. Notably, it is {\em holomorphically homogeneous} (that is, its automorphism group acts transitively on it).

The example of the hypersurface $C$ in \eqref{tube} motivated the long-lasting and somewhat intriguing study of the class $\mathcal C_{2,1}$ of real hypersurfaces $M\subset\CC{3}$ which are, in the same manner as $C$, {\em everywhere $2$-nondegenerate} (and hence have everywhere Levi rank one and a finite-dimensional local automorphism group, see e.g. \cite{ber}). Earlier work here includes the work of Ebenfelt \cite{ebenfeltC3}, where a partial (formal) normal form for  real-analytic hypersurfaces $\mathcal C_{2,1}$ was obtained (notably, this normal form already provides a finite-dimensional reduction of the equivalence problem);  the work of Ebenfelt \cite{ebenfeltduke}, where a canonical frame for a smooth $\mathcal C_{2,1}$ hypersurface $M$ was provided under certain additional assumptions; the work of Beloshapka \cite{belC3}, where the optimal dimension bound for the (full) automorphism group of a real-analytic $\mathcal C_{2,1}$ hypersurface was obtained; the work of Kaup-Zaitsev \cite{kaupzaitsev}, where the full automorphism group of the tube \eqref{tube} was finally established correctly (as a Lie group); the work of Fels-Kaup \cite{kaup2}, where the useful local rational model
\begin{equation}\Label{cone}
v=P(z,\zz,\bar z,\bar\zz),\quad \,\,\mbox{where}\,\,\,\,P(z,\zz,\bar z,\bar\zz):=\frac{|z|^2+\re(z^2\bar\zz)}{1-|\zz|^2}, \quad (z,\z,w=u+iv)\in\CC{3},
\end{equation}
 for the tube \eqref{tube} was discovered and the automorphism group of it was subsequently computed explicitely; finally, the work of Fels-Kaup \cite{kaup}, 
 %({\em Acta Math., 2008}),
 where locally homogeneous hypersurfaces of class $\mathcal C_{2,1}$ in $\CC{3}$ where classified.     More recent work here includes the work of Isaev-Zaitsev \cite{iz} 
 where the equivalence problem for smooth $\mathcal C_{2,1}$ hypersurfaces was solved; an independent solution of the equivalence problem due to Medori-Spiro \cite{ms}; 
 explicit CR-curvature computation for $\mathcal C_{2,1}$ hypersurfaces due to Pocchiola \cite{pocchiola} (for the published version see \cite{pocchiola2}); the work of Beloshapka-Kossovskiy \cite{bk} 
 on the parameterization of the stability group of a $\mathcal C_{2,1}$ hypersurface by that of the sphere $S^3\subset\CC{2}$;  a development of the work \cite{iz} and \cite{ms} in complex dimension $4$ due to Porter-Zelenko \cite{zelenko}. Notably, in the cited work, the role of the tube over the light cone \eqref{tube} as the "model"\, hypersurface for the class $\mathcal C_{2,1}$ was clearly demonstrated. In particular, it follows that maps between (germs of) two hypersurfaces of class $\mathcal C_{2,1}$ are parameterized by the stability group of the model \eqref{tube}. 

We note that the tube over the light cone admits higher dimensional generalizations 
$$C_{k,l}=\left\{(z_1,...,z_{k+l}):\,\,y_1^2+\cdots y_k^2=y_{k+1}^2+\cdots+y_{k+l}^2,\,\,\, y_j=\im z_j, \,\, y_1^2+\cdots y_k^2 \neq 0\right\}\subset\CC{k+l}$$ 
for $k,l\geq 1,\,\,k+l\geq 3.$
All the tubes $C_{k,l}$ have a Levi kernel of rank one everywhere, are everywhere $2$-nondegenerate  and their full  automorphism  group is $SO(k+1,l+1)$ (see Fels and Kaup \cite{kaup}).  The tubes $C_{k,l}$ are important in that they are boundaries of {\em Hermitian Symmetric Domains}. There exist further classes of $2$-nondegenerate hypersurfaces related to  Hermitian Symmetric Domains, for example, {\em Lie spheres}. The latter fact further motivates the study of  everywhere $2$-nondegenerate hypersurfaces and their mappings. We shall mention here recent results of Mir \cite{mirdef} and Kossovskiy-Lamel-Xiao \cite{klx} on the regularity of mappings into $2$-nondegenerate hypersurfaces, and their applications in studying proper holomorphic maps of Hermitian Symmetric Domains due to Xiao-Yuan \cite{xiao}.   For more on the geometry of holomorphic maps between Hermitian Symmetric Domains, we refer to the Introduction in \cite{xiao} and references therein. 

\subsection{Overview of the results} The above mentioned substantial study of everywhere $2$-nondegenerate hypersurfaces still leaves completely open the following aspect of the problem: {\em extend Moser's normal form theory to the class of (real-analytic) everywhere $2$-nondegenerate hypersurfaces}.   
That is, we aim for {\em a complete convergent normal form} for  real-analytic everywhere $2$-nondegenerate hypersurfaces. We shall explain here breifly the main difficulty in extending Moser's homological approach and obtaining a complete normal form. Namely, 
as follows from the infinite Catlin multitype property, there is {\em no} choice of positive weights for the coordinates in $\CC{3}$ relevant to everywhere $2$-nondegenerate hypersurfaces such that each hypersurface becomes a perturbation of a holomorphically nondegenerate polynomial model.  In this way, $2$-nondegenerate hypersurfaces can {\em not} be treated in a  manner of Chern-Moser (and subsequent work) for obtaining their complete normal form. 

The main goal of this paper is to provide a new technique for overcoming difficulties as above and, as an outcome, obtain the desired complete normal form for everywhere $2$-nondegenerate hypersurfaces in $\CC{3}$. The normal form, accordingly, is defined uniquely up to the stability group of the model: the tube over the light cone \eqref{tube} (or, equivalently, its local rational realization \eqref{cone}). Notably, we are able to present any hypersurface of class     $\mathcal C_{2,1}$ as a perturbation of the rational model, and then set up, in the spirit of Moser, a {\em homological procedure} for investigating maps between two hypersurfaces. The latter is accomplished by a choice of weights for the coordinates in $\CC{3}$ where the coordinate corresponding to the Levi kernel gets weight $0$. Somewhat surprisingly, even though this approach leads to natural difficulties (for example, "weighted polynomials"\, are not polynomials anymore but power series), we are still able to reduce the space of mappings between hypersurfaces to studying the kernel of an appropriate {\em homological operator}. The   kernel of the latter operator is precisely the automorphism algebra of the model (given below). We shall note that such approach shares certain distinctive traits with that due to Huang-Yin in \cite{hy} (where a certain non-standard choice of weights served as a key to deal with the infinite-dimensionality of the automorphism group of the weighted homogeneous model over there).    

Similarly to the situation of Chern-Moser, the convergence is achieved by employing {\em chains}: distinguished curves in $M$, pointwise transverse to the complex tangent 
and being mapped by normalizing transformations into the standard "vertical"\, curve (see \eqref{Gamma} below). 
Since orbits of the linear part of the stabilizer do {\em not} act transitively on transverse directions anymore (unlike the Levi-nondegenerate situation), 
possible directions of chains form in the tangent space a certain {\em canonical cone} (given in appropriate coordinates by \eqref{cancone}). 
The construction of chains and their properties are discussed in detail in Section 3.2. We shall also mention that certain (degenerate) chains
were previously used by Kossovskiy-Zaitsev in \cite{generic},\cite{cmhyper} for proving convergence of normal forms of finite type hypersurfaces (obtained by Kol\'a\v r in \cite{kol05}). 
Chains for submanifolds of high codimension were also used in the work \cite{es} of Ezhov-Schmalz. As discussed below, for constructing the chains we use a quite recent original approach due to Zaitsev \cite{zaitsevnf}.

We shall emphasize that, to the best of our knowledge, the present paper gives the first  development of Moser's homological method to the class of infinite Catlin multitype hypersurfaces.  

As an application of  our theory, we obtain, similarly to the situation in Chern-Moser's theory, a characterization of {\em sphericity} for an everywhere $2$-nondegenerate hypersurface (i.e., the property of being equivalent to the model). The sphericity then amounts to the pointwise vanishing of two specific coefficients in a normal form of the hypersurface at a point (see \autoref{main2} below).  

Finally, we are able to apply the normal form for describing explicitly the {\em moduli space} of (real-analytic) everywhere $2$-nondegenerate hypersurafaces, 
considered up to a local biholomorphic equivalence (\autoref{main3} below). 
A certain difficulty in applying the normal form for describing the moduli space is that our normal form space describes in 
fact a larger class of hypersurfaces than that of {everywhere} $2$-nondegenerate ones. We overcome this difficulty 
by selecting a certain {\em distinguished part} of the normal form, which uniquely determines the rest of the normal form under the assumption of everywhere $2$-nondegeneracy.     In turn, the latter comes from the study of a certain ``mixed''\, Cauchy problem with the Cauchy data on a ``cross''.   Details are provided below in Section 3.4.

\subsection{The normal form}

Let $M\subset\CC{3}$ be an everywhere $2$-nondegenerate hypersurface. The tangent bundle $TM$ is endowed then with the canonical subbundles $T^{\CC{}}M$ and $K$, where $K_p$ is the Levi kernel at $p$. Accordingly, the quotient bundles  $TM/T^{\CC{}}M$ and $TM/K$ are well defined. 
%In view of that, for a CR-diffeomorphism $H:\,M\lr M^*$ between manifolds under consideration, the induced  maps $H':\,TM/T^{\CC{}}M\lr TM*/T^{\CC{}}M*$ and $H'':\,TM/K\lr TM*/K*$ are well defined. 

Further, let us introduce the space $\mathcal N$ of real-valued formal power series $$\Phi(z,\z,\bar z,\bar \z, u)=\sum_{k+l+\alpha+\beta\geq 5}\Phi_{kl\alpha\beta}(u)z^k\z^l\bar z^\alpha\bar\z^\beta$$
satisfying, in addition,
\begin{equation}\Label{nspace}
\begin{aligned}
&\Phi_{kl00}=\Phi_{kl10}=\Phi_{kl20}=0,\,\,k,l\geq 0;\\ 
&\Phi_{3001}=\Phi_{4001}=\Phi_{3011}=\Phi_{4011}=\Phi_{3030}=0.
\end{aligned}
\end{equation}
Here $(z,\zeta,w=u+iv)$ are the coordinates in $\CC{3}$.
\begin{definition}\Label{innf}
We say that a (formal or analytic) hypersurface $M\subset\CC{3}$ passing through $0$ {\em is in normal form}, if it is given by a defining equation 
\begin{equation}\Label{model+}
v=P(z,\z,\bar z,\bar\z)+\Phi(z,\z,\bar z,\bar \z, u),
\end{equation}
\end{definition}
\noindent where $P$ is as in \eqref{cone} and $\Phi\in\mathcal N$. 

\smallskip

We now formulate our main result. We make use of the local representation \eqref{germ} below for an everywhere $2$-nondegenerate hypersurface.

\begin{theorem}\label{main}
Let $M$ be a real-analytic everywhere $2$-nondegenerate hypersurface in $\CC{3}$, and $p\in M$. Then, there exists a biholomorphic transformation $H:\,(\CC{3},p)\lr(\CC{3},0)$ mapping $M$ into a hypersurface in normal form. A normalizing transformation $H$ is determined uniquely by the action of its differential $dH_p$ on the quotient space $T_pM/K_p$  and the transverse second order derivative at $p$ of the transverse component of $H$. 

In turn,  in any coordinates \eqref{germ}, a normalizing transformation is unique up to the right action of the $5$-dimensional  stability group of the model \eqref{cone}.
\end{theorem}
The stability group of the model \eqref{cone} is described in detail in Section 2.1 below.
\begin{remark}\label{2nondegnf}
We shall emphasize once more that the class of  hypersufaces in normal form \eqref{nspace} is larger than that of {everywhere} $2$-nondegenerate hypersurfaces in normal form \eqref{nspace}. However, as shown in \autoref{main3} below, it is possible to select a distinguished part of the normal form which, first of all, can be chosen arbitrary, and second, determines one and only one everywhere $2$-nondegenerate hypersurface with the given distinguished part of the normal form. 
\end{remark}

\subsection{Applications of the normal form} In this section, we provide two applications of the complete convergent normal form provided in \autoref{main}.

First of all, by combining \autoref{main} with the result of Pocchiola \cite{pocchiola}, we obtain a criterion for the {\em sphericity} of a general hypersurface, i.e. its local equivalence  to the model: the tube over the light cone.

\begin{theorem}\Label{main2}
Let $M\subset\CC{3}$ be a real-analytic everywhere $2$-nondegenerate hypersurface, and $p\in M$. Then $M$ is locally biholomorphic near $p$ to the tube over the light cone \eqref{tube} if and only if for every point $q\in M$ nearby $p$ we have: 
\begin{equation}\Label{sphericity}
\Phi_{3002}(0)=\Phi_{5001}(0)=0
\end{equation}
for the coefficients of some (and hence any) normal form at $q$.
\end{theorem}
\autoref{main2} is a direct analogue in the $2$-nondegenerate setting of Chern-Moser's theorem in \cite{chern} on the characterization of sphericity of a Levi-nondegenerate hypersurface via coefficients of its normal form.
\begin{remark}\Label{smooth}
The assertion of \autoref{main2} can be extended to the class of merely {\em smooth} everywhere $2$-nondegenerate submanifolds $M$. In this case, one has to employ the {\em formal} normal form for $M$, which is still available in the smooth case. 
\end{remark}

We now describe our second application of the normal form. Since the inception of the notion of finite nondegeneracy (originally introduced, as mentioned above, by Baouendi-Huang-Rothschild in \cite{bhr}), it has been an open problem to describe the moduli space of (real-analytic) everywhere $k$-nondegenerate hypersurfaces in $\CC{N}$ for various $k,N$. To the best of our knowledge, no such description has been available till present in any dimension. Our normal form allows for solving the latter problem for $2$-nondegenerate hypersurfaces in $\CC{3}$. Namely, fix   a real-analytic everywhere $2$-nondegenerate hypersurface $M\subset\CC{3}$ in normal form \eqref{model+} at the origin. Consider then, for the normal form, the  sum of all terms in $\Phi$ which are not divisible by $\z\bar \z$. The latter has the form 
\begin{equation}\label{disting}
\Phi(z,\z,\bar z,0,u)+\Phi(z,0,\bar z,\bar\z,u)-\Phi(z,0,\bar z,0,u).%=:2\re\chi(z,\z,\bar z,u)
\end{equation}
Analytic at the origin expressions \eqref{disting} are in one-to-one correspondence with analytic the  near the origin functions
$$\chi(z,\z,\bar z,u)=\sum_{k,l,\alpha\geq 0}\chi_{kl\alpha }(u)z^k\z^l\bar z^\alpha$$
appearing as
$$\chi(z,\z,\bar z,u):=\Phi(z,\z,\bar z,0,u)=\sum_{k,l,\alpha\geq 0}\Phi_{kl\alpha 0}(u)z^k\z^l\bar z^\alpha.$$ 

\smallskip

\noindent The latter functions $\chi$ satisfy, in case $\Phi$ is in normal form, the conditions:
\begin{equation}\label{Dspace}
\begin{aligned}
&\chi(z,0,\bar z,u)\in\RR{},\\ 
&\chi_{kl0}=\chi_{kl1}=\chi_{kl2}=0,\,\,k,l\geq 0,\\ 
&\chi_{013}=\chi_{014}=\chi_{113}=\chi_{114}=\chi_{303}=0
\end{aligned}
\end{equation}
 (coming respectively from the reality requirement $\Phi(z,\z,\bar z,\bar\z,u)\in\RR{}$ and the normal form conditions \eqref{nspace}).  

%\begin{equation}\label{series}
%\begin{aligned}
%\chi(z,\z,\bar z,u):=\Phi(z,\z,\bar z,0,u)=\sum_{k+l+\alpha+\geq 5}\Phi_{kl\alpha 0}(u)z^k\z^l\bar z^\alpha, \\ 
%\tau(z,\z,\bar z,u):=\Phi_{\bar\z}(z,\z,\bar z,0,u)=\sum_{k+l+\alpha+\geq 4}\Phi_{kl\alpha 0}(u)z^k\z^l\bar z^\alpha.
%\end{aligned}
%\end{equation}
%If expanding $$\chi(z,\z,\bar z,u)=\sum_{k+l+\alpha+\geq 5}\chi_{kl\alpha}(u)z^k\z^l\bar z^\alpha$$ and similarly for $\tau$, we have:
%\begin{equation}\label{realty}
%\chi_{\alpha lk}(u)=\overline{\chi_{kl\alpha}(u)}, \quad \chi_{kl0}(u)=\chi_{kl1}(u)=\chi_{kl2}(u)=0,\,\,k,l\geq 0,\,\,
%\end{equation}
%We call the linear in $\bar\z$ function
%\begin{equation}\label{principal}
%\chi(z,\z,\bar z,u)+\tau(z,\z,\bar z,w)\bar\z
%\end{equation}
\begin{definition}
We call the function \eqref{disting}  {\em the distinguished part of the normal form of $M$ at $p$}. 
 We also denote by $\mathcal D$ the respective linear functional space \eqref{Dspace}.
 \end{definition}
 
 We are now in the position to describe the moduli space under discussion.
 
\begin{theorem}\label{main3}
The natural map $\pi$ assigning to a real-analytic everywhere $2$-nondegenerate hypersurface $M\subset\CC{3}$ in normal form \eqref{nspace} the distinguished part of the normal form is a bijection onto  the space $\mathcal D$.

In this way, the moduli space of real-analytic everywhere $2$-nondegenerate hypersurfaces considered up to a local biholomorphic equivalence carries the structure $$\mathcal D/G,$$ where $G$ is the ($5$-dimensional) stability subgroup in the automorphism group of the rational model \eqref{cone}.  
\end{theorem}
The action of the group $G$ on the space $\mathcal D$ can be  viewed  from the normalization  procedure in Section 2.3 (while the group $G$ is described, once again, in Section 2.1). 

\bigskip

{\it Acknowledgements:}{
We would like to thank Jan Gregorovic for his careful reading and useful remarks on an early version of this paper.} 

\section{A formal normal form}

\subsection{The light cone and its automorphisms} 

We shall now describe in detail the automorphisms of the local rational model \eqref{cone} for the tube over the light cone \eqref{tube}
(obtained by Fels-Kaup in \cite{kaup2}).

According to \cite{kaup2}, the cone \eqref{tube} is locally bi-rationally equivalent to the real-algebraic (rational) hypersurface 
\eqref{cone}.
In \cite{kaup2}, the infinitesimal automorphism algebra of \eqref{cone} is described using a grading by complex integers. Our approach is different though. We introduce the following choice of weights playing a crucial role in our construction: 
\begin{equation}\Label{weights}
[z]=[\bar z]=1,\,\,[\zeta]=[\bar \zeta]=0,\,\,[w]=[\bar w]=2
\end{equation}
The rational model \eqref{cone} is then a {\em homogeneous hypersurface} with respect to the choice of weights \eqref{weights}.

We shall recall that the {\em infinitesimal automorphism algebra } of a CR-submanifold $M\subset\CC{N}$ at a point $p\in M$ is the Lie algebra $\mathfrak{hol}(M,p)$  of holomorphic vector fields $$f_1(z)\frac{\partial}{\partial z_1}+\cdots+f_N(z)\frac{\partial}{\partial z_N}$$ with coefficients $f_j$ holomorphic in an open neighborhood of $p$ in $\CC{N}$ such that their real part is  tangent to $M$ at every point $q\in M$. The vector fields in $\mathfrak{hol}(M,p)$ are precisely the ones generating flows of local biholomorphic automorphisms of the germ $(M,p)$. 

Let us now employ the weight system \eqref{weights} to present the infinitesimal automorphism algebra $\g$ of  \eqref{cone} as a graded Lie algebra. Let us set:
\begin{equation}\Label{grading}
[z]=1,\,\,[\zeta]=0,\,\,[w]=2,\,\,\left[\frac{\partial}{\partial z}\right]=-1,\,\,\left[\frac{\partial}{\partial \zz}\right]=0,\,\,\left[\frac{\partial}{\partial w}\right]=-2.
\end{equation}
Now the results in \cite{kaup2} can be interpreted as presenting $\g$ in the form:
$$\g=\g_{-2}\oplus\g_{-1}\oplus\g_{0}\oplus\g_{1}\oplus\g_{2},$$
where the negatively graded components are
\begin{equation}\Label{g-}
\begin{aligned}
\g_{-2}=&\mbox{span}\,\left\{\frac{\partial}{\partial w}\right\},\\
\g_{-1}=&\mbox{span}\,\left\{(1-\zz)\frac{\partial}{\partial z}+2iz\frac{\partial}{\partial w},\,\,i(1+\zz)\frac{\partial}{\partial z}+2z\frac{\partial}{\partial w}\right\},\qquad\qquad\qquad\qquad\qquad
\end{aligned}
\end{equation}
the zero component is split as
$\g_{0}=\g_0^c\oplus\g_0^s$
with
\begin{equation}\Label{g0}
\begin{aligned}
\g_0^c=&\mbox{span}\,\left\{z\frac{\partial}{\partial z}+2w\frac{\partial}{\partial w},\,\,iz\frac{\partial}{\partial z}+2i\zz\frac{\partial}{\partial \zz}\right\},\\
\g_0^s=&\mbox{span}\,\left\{  -z\zz\frac{\partial}{\partial z}+ (1-\zz^2)\frac{\partial}{\partial \zz}+iz^2\frac{\partial}{\partial w},\,\,iz\zz\frac{\partial}{\partial z}+i(1+\zz^2)\frac{\partial}{\partial \zz}+z^2\frac{\partial}{\partial w}\right\},
\end{aligned}
\end{equation}
and the positively graded components are
\begin{equation}\Label{g+}
\begin{aligned}
\quad\g_{1}=&\mbox{span}\,\left\{(z^2+iw+i\zz w)\frac{\partial}{\partial z}+2(z+z\zz)\frac{\partial}{\partial \zz}+2zw\frac{\partial}{\partial w},\right.\\
\qquad\mbox{}&\left.\qquad\qquad\qquad \qquad\qquad\quad(iz^2+w+\zz w)\frac{\partial}{\partial z}+2i(z\zz-z)\frac{\partial}{\partial \zz}+2izw\frac{\partial}{\partial w}\right\},\\
\g_{2}=&\mbox{span}\,\left\{zw\frac{\partial}{\partial z}-iz^2\frac{\partial}{\partial \zz}+w^2\frac{\partial}{\partial w}\right\}.
\end{aligned}
\end{equation}
We conclude that the stability subalgebra $\mathfrak h=\{X\in\g:\,\,X(0)=0\}\subset\g$ is decomposed as
\begin{equation}\Label{isotropy}
\mathfrak h=\g_0^c\oplus\g_1\oplus\g_2.
\end{equation}
The stability group $G$ of the rational model \eqref{cone} is generated by the algebra \eqref{isotropy} (i.e. it is the exponential image of \eqref{isotropy}). Its subgroup $G_0^c$ generated by the component $\g_0^c$ consists of the scalings
\begin{equation}\Label{scalings}
z\mapsto re^{i\varphi}z,\quad \zz\mapsto e^{2i\varphi}\zz, \quad w\mapsto r^2w, \quad r>0,\,\,\varphi\in\RR{}.
\end{equation}
The subgroup $G_+$ generated by the subalgebra $\g_+:=\g_1\oplus\g_2$ consists of certain rational transformations. As the respective rational expressions are cumbersome and are not used in the paper, we do not provide them here.

\subsection{The light cone as a model} Recall that, according to  Ebenfelt \cite{ebenfeltC3}, in appropriate local holomorphic coordinates a germ of an everywhere Levi degenerate but holomorphically nondegenerate hypersurface $M\subset\CC{3}$ can be represented as:
\begin{equation}\Label{germ}
v=|z|^2+\sum_{k\geq 3} Q^k(z,\zeta,\bar z,\bar \zeta,u), 
\end{equation}
where all $Q^k$ are homogeneous polynomials of weight $k$ with respect to the grading system
\begin{equation}\Label{initweights}
[z]=[\zeta]=[\bar z]=[\bar \zeta]=1,\,\,[w]=[\bar w]=2,
\end{equation}
and, in addition,
\begin{equation}\Label{Q3}
Q^3=\frac{1}{2}(z^2\bar\zeta+\bar z^2\zeta).
\end{equation}
Furthermore, as shown in \cite{bk}, we may assume (after a formal coordinate change) that all terms of the kind 
\begin{equation}\Label{killed}
z^k\zz^l\bar z^0\bar\zz^0u^m, \quad z^k\zz^l\bar z^1\bar\zz^0u^m,\,(k,l)\neq (1,0), \quad z^k\zz^l\bar z^2\bar\zz^0u^m,\,(k,l)\neq (0,1)
\end{equation}
and their conjugated are not present in \eqref{germ}. We will furtheremore show in Section 3 that coordinates \eqref{killed} can be in turn chosen holomorphic. 
\begin{definition}
In what follows, a (formal) hypersurface \eqref{germ}, satisfying \eqref{Q3} and the condition of absence of terms \eqref{killed} is called {\em prenormalized}.
\end{definition}

As discussed in the Introduction, the key point of this paper is using, in contrast to \eqref{initweights}, the weight system \eqref{weights}, which reflects the (infinite)
Catlin multitype of the manifold. We make at this point the following
\begin{convention}
We say that a formal series in $z,\zz,\bar z,\bar\zz,u$ {\em has weight $k\geq 0$} in the grading \eqref{weights}, if each Taylor polynomial of it has the weight $k$ in the grading \eqref{weights}.
\end{convention}

Our immediate goal is to show that, with respect to the weights \eqref{weights}, any everywhere Levi degenerate hypersurface $M\subset\CC{3}$ is a perturbation of
(the rigid model) the light cone \eqref{cone}. 

Recall that $P(z,\zz,\bar z,\bar\zz)$ denotes the right hand side of \eqref{cone}.

\begin{proposition}\Label{goodperturb}
Let $M$ be a (formal or analytic) prenormalized hypersurface \eqref{germ}. Then the defining equation of $M$ can be written as
\begin{equation}\Label{perturb}
v=P(z,\zz,\bar z,\bar\zz)+\sum_{k\geq 3}\Phi_k(z,\zeta,\bar z,\bar \zeta,u),
\end{equation}
where each $\Phi_k$ has weight $k$ with respect to the grading \eqref{weights}.
\end{proposition}
\begin{proof}
It is convenient to switch to the complex defining equation 
$$w=\theta(z,\bar z,\bar w)$$
in \eqref{germ} (see \cite{ber}), and write the result as
\begin{equation}\Label{germc}
w=\bar w+2i\Bigl(|z|^2+\sum_{k\geq 3} Q^k(z,\zeta,\bar z,\bar \zeta,\bar w)\Bigr)
\end{equation}
(the polynomials $Q^k$ in \eqref{germ} and \eqref{germc} are in principle different, but we keep the same notations for simplicity).
Let us denote by $Q^k_0,\,k\geq 3,$ the polynomials $Q^k$, as in \eqref{germc}, corresponding to the model hypersurface \eqref{cone}. That is, $$Q^{2j}_0=|z|^2|\zz|^{2j-2},\,j\geq 2,$$ and $$Q^{2j+1}_0=\frac{1}{2}(z^2\bar\zeta+\bar z^2\zeta)|\zz|^{2j-2},\,j\geq 1.$$
The assertion of the proposition reads then as
\begin{equation}\Label{Qk}
Q^k=Q^k_0+O(|z|^3)+O(|z||\bar w|)+O(|\bar w|^2).
\end{equation}
Recall that the Levi determinant in terms of the complex defining equation equals to 
$$\begin{vmatrix}
\theta_{\bar z} & \theta_{\bar \zz} & \theta_{\bar w} \\
\theta_{z\bar z} & \theta_{z\bar \zz} & \theta_{z\bar w} \\
\theta_{\zz\bar z} & \theta_{\zz\bar \zz} & \theta_{\zz\bar w} 
\end{vmatrix}$$
(see \cite{ber}). 
Accordingly, the uniform Levi degeneracy condition for $M$ reads as 
\begin{equation}\Label{det}
\begin{vmatrix}
2iz+2i\sum Q^k_{\bar z} & 2i\sum Q^k_{\bar \zz}  & 1+2i\sum Q^k_{\bar w}  \\
2i+2i\sum Q^k_{z\bar z} & 2i\sum Q^k_{z\bar \zz}  & 2i\sum Q^k_{z\bar w} \\
2i\sum Q^k_{\zz\bar z} & 2i\sum Q^k_{\zz\bar \zz}  & 2i\sum Q^k_{\zz\bar w}
\end{vmatrix}\equiv 0.
\end{equation}
We then inspect a series of identities obtained by collecting in \eqref{det} terms of a fixed weight $k-2,\,k\geq 3$ {\em with respect to the standard grading \eqref{initweights}}. For $k=3$ such an identity  holds in view of \eqref{Q3}. However, for $k\geq 4$ we obtain non-trivial conditions involving the polynomials $Q^3,...,Q^k$. It is not difficult to see by expanding the determinant  that these conditions have the form:
\begin{equation}\Label{iterative}
-4Q^k_{\zz\bar\zz}=\cdots,
\end{equation} 
where dots stand for a polynomial expression in the variable $z$ and the polynomials $Q^j,\,j<k$. 

We will prove \eqref{Qk} by induction in $k\geq 3$. For $k=3$ it holds in view of the above discussed formula for $Q^3$. For the induction step, we assume that \eqref{Qk} holds for all $j\leq k$, and consider the identity obtained by collecting in \eqref{det} terms with weight $k-1$. In view of \eqref{iterative}, this identity has the form
\begin{equation}\Label{iterative1}
-4Q^{k+1}_{\zz\bar\zz}=\cdots
\end{equation} 

We claim that, in fact, \eqref{iterative1} has the more specific form 
\begin{equation}\Label{iterative2}
-4Q^{k+1}_{\zz\bar\zz}=-4\left(Q^{k+1}_0\right)_{\zz\bar\zz}+\cdots,
\end{equation} 
where dots stand for terms of the form 
\begin{equation}\Label{remainder}
O(|z|^3)+O(|z||\bar w|)+O(|\bar w|^2).
\end{equation}
To prove the claim, we apply the representation \eqref{Qk} for $Q^j$ with $j\leq k$ in the right hand side of \eqref{iterative1} (which is possible due to the induction assumption) and conclude that we have two kinds of terms arising: those arising from $Q^j_0$ with $j\leq k$ only (that is, terms obtained by multiplying that in the derivatives of $Q^j_0$), and all the others. 

For the first group of terms, we note that they must have simply the form $\left(Q^{k+1}_0\right)_{\zz\bar\zz}$. This is concluded from considering the equation \eqref{det} for the model hypersurface \eqref{cone} itself.

To analyze the second group of terms, we first observe that the $(1,2)$ and the $(3,2)$ entries of the determinant  \eqref{det} already have the form
\eqref{remainder}. This follows from the induction assumption and the invariance of \eqref{remainder} under differentiation in $\z,\bar\z$.
Hence it is enough to consider only the terms coming from the two "diagonal" products ($a_{11}a_{22}a_{33}$ and $a_{31}a_{22}a_{13}$) within the expansion of the determinant \eqref{det}. By the definition of the second group of terms under consideration and the induction assumption, each such term coming from the two diagonal products must contain at least one derivative of an expression $$O(|z|^3)+O(|z||\bar w|)+O(|\bar w|^2)$$ (as in \eqref{remainder}), and the derivation in the group of variables $z,\bar z,\bar w$ happens at most once. Furtheremore, since the total weight is $k-1$, the latter derivative gets necessarily multiplied by at least one derivative of either the same kind of expression, or an expression $O(|z|^2)$ (as follows from \eqref{remainder} and the induction assumption), and the derivation in the group of variables $z,\bar z,\bar w$ again happens at most once. The product of two such derivatives automatically has the form $O(|z|^3)+O(|z||\bar w|)+O(|\bar w|^2)$,  and this proves the claim.

It remains to deal with the identity \eqref{iterative2} that holds true according to the claim. Integrating this identity in $\zz,\bar\zz$, we get:
\begin{equation}\Label{almost}
Q^{k+1}=Q^{k+1}_0+R_1+R_2,
\end{equation}
where a remainder $R_1$ has the form \eqref{remainder}, while $R_2$ satisfies $\bigl(R_2\bigr)_{\zz\bar\zz}=0$. In view of the absence of the terms \eqref{killed}, this gives $R_2\equiv 0$, and implies \eqref{Qk}. Proposition is proved now.

\end{proof}

As the final outcome of this subsection, we may restrict our considerations to (formal or analytic) hypersurfaces of the kind \eqref{perturb} 
(that is, to "good" perturbations of the light cone \eqref{cone} in the weight system \eqref{weights}). That is why in what follows we use the weight system \eqref{weights} only. 

\subsection{Normalization of initial terms of a CR-map} Let $M,M^*$ be two (formal or analytic) prenormalized hypersurfaces of the kind \eqref{perturb}, and 
$$H=(f,g,h):\,\,(M,0)\lr(M^*,0)$$
be a (formal) invertible holomorphic map between them. Let us first note that the representation \eqref{perturb} is invariant under the group of scalings \eqref{scalings}.

The main goal of this subsection is to prove the following
\begin{proposition}\Label{initterms}
There exists a scaling $\Lambda$, as in \eqref{scalings}, such that $H$ can be decomposed as $H=\tilde H\circ\Lambda$, where the (formal) map $\tilde H$ has the form:
\begin{equation}\Label{normalmap}
z\mapsto z+f_2+f_3+\cdots, \quad \zz\mapsto \zz+g_1+g_2+\cdots, \quad h=w+h_3+h_4+\cdots,
\end{equation}
where $f_j,g_j,h_j$ are formal power series of weight $j$ with respect to the grading \eqref{weights}.
\end{proposition}

\begin{proof}
The basic identity for the map $H$ gives:
\begin{equation}\Label{basic}
\begin{aligned}
\im h=P(f,g,\bar f,\bar g)+&\sum_{j\geq 3}\Phi^*_j(f,g,\bar f,\bar g,\re h), \\ 
f=f(z,\zz,w),\,g=g(z,\zz,w),\,h=h(z,\zz,w),\, & w=u+iP(z,\zz,\bar z,\bar\zz)+i\sum_{j\geq 3}\Phi_j(z,\zeta,\bar z,\bar\zz,u).
\end{aligned}
\end{equation}

Comparing in \eqref{basic} terms of weight $0$, we get 
\begin{equation}\Label{weight0}
\im h_0(\zz)=P(f_0(\zz),g_0(\zz),\bar f_0(\bar \zz),\bar g_0(\bar\zz))+\sum_{j\geq 3}\Phi^*_j(f_0(\zz),g_0(\zz),\bar f_0(\bar\zz),\bar g_0(\bar\zz),\re h_0(\zz))
\end{equation}
(note that weight $0$ components can depend on $\zz$ only). 
Recall that both hypersurfaces are prenormalized, hence the right hand side in \eqref{basic} does not contain pluriharmonic terms  and we conclude that $h_0=0$. (In what follows, we mean by pluriharmonic terms all terms which are real parts of holomorphic functions in $z,\zeta$, where the real variable $u$ is treated as a parameter). We next claim that $f_0=0$. Indeed, asssume that $m:=ord_0\,f_0(\zz)<\infty$. Then the hermitian term in $P$ gives in the right hand side of \eqref{weight0}  a non-zero term with $\zz^m\bar\zz^m$. On the other hand, no other terms in the right hand side of \eqref{weight0} contribute to  $\zz^m\bar\zz^m$ (this follows from $h_0=0$ and from the fact that all $\Phi^*_j$ have weight at least $3$). Thus we conclude that $f_0=0$. Finally, since the map $H$ is invertible, we can claim that $g_0(\zz)$ has a nonzero linear part. 

We next collect terms of weight $1$ in \eqref{basic}. Since we have $f_0=h_0=0$, this simply gives: $\im h_1(z,\zz)=0$, so that $h_1=0$. For the components $f_1,g_1$ we observe that they have respectively the form $zF_1(\zz),zG_1(\zz)$. The invertibility of the map also gives $F_1(0)\neq 0$ and $h_w(0,0,0)\neq 0$.

As the last step, we collect in \eqref{basic} terms of weight $2$. This gives:
\begin{equation}\Label{weight2}
\begin{aligned}
\im h_2=P(f_1,g_0,\bar f_1,\bar g_0) ,\,\,\,
w=u+iP(z,\zz,\bar z,\bar\zz).
\end{aligned}
\end{equation}
We conclude that the invertible map
$$(z,\zz,w)\lr \bigl(f_1(z,\zz),g_0(\zz),h_2(z,\zz,w)\bigr)$$
is an automorphism of the light cone \eqref{cone}. In view of the explicit description given by \eqref{g0},\eqref{g+},\eqref{isotropy}, this implies (integrating $\mathfrak g_0^c$):
$$f_1=\lambda z,\,\,g_0=\mu\zz,\,\,h_2=\nu w, \quad \lambda,\mu,\nu\neq 0.$$ 
Substituting the latter into \eqref{weight2} and comparing terms with $z^2\bar\zz$, we obtain:
\begin{equation}\Label{munu}
\mu=e^{2i\varphi},\quad \nu=r^2,\,\,\mbox{where}\,\,\lambda=re^{i\varphi},\,\,r>0,\varphi\in\RR{}.
\end{equation}
The identity \eqref{munu} (together with the earlier obtained description of the weighted components of the map $H$) implies the assertion of the proposition.
\end{proof}
We end this subsection by noting that any map $H$, as in \eqref{normalmap}, can be further decomposed in a unique way as \begin{equation}\Label{factord}
H=\tilde H\circ\psi,
\end{equation}
 where $\psi$ is an automorphism of the light cone \eqref{cone} belonging to the subgroup \eqref{g+},  and the map $\tilde H=(z+f,\zz+g,w+h)$ has the form \eqref{normalmap} and satisfies, in addition,
\begin{equation}\Label{specialmap}
f_{zz}=0,\quad \re h_{ww}=0.
\end{equation}
That is why, in what follows, {\em for the normalization procedure we may consider only maps of the special form \eqref{specialmap}}. Collecting together the normalization conditions for the map, we get the factorization
\begin{equation}\Label{factored}
H=\tilde H\circ\psi,
\end{equation}
where $\psi$ is an element of the stability group  of the light cone \eqref{cone},  and the map $\tilde H=(z+f,\zz+g,w+h)$ has the form \eqref{normalmap} and satisfies, in addition, \eqref{specialmap}. In fact, we will show later that a normalizing transformation $\tilde H$, as in \eqref{factored}, is already {\em unique}, so that \eqref{factored} can be seen as an {\em action of the stability group of the light cone \eqref{cone} on the normal forms}.

\subsection{Homological operator and formal normal form}
Let us consider a map $H:\,\,(M,0)\lr(M^*,0)$ between two prenormalized (formal) hypersurfaces $M,M^*$, which is decomposed as in \eqref{normalmap} and satisfies, in addition, \eqref{specialmap}. Let us consider, for each fixed $m\geq 3$, the equation obtained by collecting in the basic identity \eqref{basic} all terms of weight $m$. In view of \eqref{perturb},\eqref{normalmap}, it is not difficult to see that such an identity must have the following form:
\begin{equation}\Label{homolog}
\re\Bigl(ih_m(z,\zz,w)+2f_{m-1}(z,\zz,w)P_z+2g_{m-2}(z,\zz,w)P_{\zz}\Bigr)|_{w=u+iP}=\Phi^*_m-\Phi_m+\cdots,
\end{equation}
where $\Phi^*_j,\Phi_j$ are as in \eqref{perturb} and dots stand for a polynomial expression in: (i) $\Phi^*_j,\Phi_j$ with $j<m$ and their derivatives of order $\leq m-1$; (ii) the collections $(f_{j-1},g_{j-2},h_j)$ with $j<m$ and their derivatives of order $\leq m-1$; (iii) the local coordinates $z,\zz,\bar z,\bar\zz,u$. This leads to the consideration of the {\em homological operator}
\begin{equation}
\mathcal L(f,g,h):=\re\Bigl(ih(z,\zz,w)+2f(z,\zz,w)P_z+2g(z,\zz,w)P_{\zz}\Bigr)|_{w=u+iP},
\end{equation}
defined on the linear space $\mathcal V$ of tuples $\bigl(f(z,\zz,w),g(z,\zz,w),h(z,\zz,w)\bigr)$ of the kind
\begin{equation}\Label{source}
f=f_2+f_3+\cdots,\,\,g=g_1+g_2+\cdots,\,\,h=h_3+h_4+\cdots
\end{equation}
satisfying, in addition, \eqref{specialmap}, and valued in the space $\mathcal W$ of series $\Phi(z,\zz,\bar z,\bar\zz,u)$ of the kind
$$\Phi=\Phi_3+\Phi_4+\cdots.$$

We shall recall now that the normal form space $\mathcal N$, as in \eqref{nspace} is given in more detail (using Proposition \ref{goodperturb}) by the conditions:
\begin{equation}\Label{Nspace}
\Phi_{\a \b 00} = \Phi_{\a \b 10} = \Phi_{\a \b 2 0} = 0
\end{equation}
for all $\a, \b$, then
\begin{equation}\Label{Nspace2}
\Phi_{4011} = \Phi_{3001} = \Phi_{4001} = 0
\end{equation}
and finally
\begin{equation}\Label{Nspace3}
\Phi_{3030} =  \Phi_{3011} = 0.
\end{equation}
Let us then take into consideration the range $\mathcal R$ of the operator $\mathcal L$. We shall prove the following
\begin{proposition}\Label{directsum}
The operator $\mathcal L$ is injective on $\mathcal V$. Moreover, the target space $\mathcal W$ described above can be decomposed as the direct sum
$$\mathcal W=\mathcal R\oplus\mathcal N,$$
where $\mathcal N\subset\mathcal W$ is the above normal form space.
\end{proposition}
\begin{proof}
The assertion of the proposition can be reformulated like that: an equation
\begin{equation}\Label{homeq}
2\mathcal L(f,g,h)=\Psi,\,\,\,\Psi\in\mathcal W
\end{equation}
has a unique solution in $\mathcal V$, modulo an element of the space $\mathcal N$ standing in the right hand side of \eqref{homeq} (which is obtained by the projection of $\Psi$ onto $\mathcal N$ parallel to  $\mathcal R$).

To solve an equation \eqref{homeq}, let us expand
\begin{equation}\Label{expandf}
f(z,\zz,u+iP)=f(z,\zz,u)+f_w(z,\zz,u)iP-\frac{1}{2}f_{ww}(z,\zz,u)P^2-\frac{i}{6}f_{www}(z,\zz,u)P^3+\cdots,
\end{equation}
and similarly for $g,h$. Further, we use the expansion
\begin{equation}\Label{fk}
f(z,\zz,u)=\sum_{k,l\geq 0}f_{kl}(u)z^k\zz^l,
\end{equation}
and similarly for $g,h$.
We now substitute \eqref{expandf},\eqref{fk} into \eqref{homeq}, and apply   {\em all} the linear functionals, annihilating the above subspace $\mathcal N$, to the resulting identity. This amounts to collecting the respective terms in \eqref{homeq}, which we do step-by-step. We write down the corresponding homogeneous equations. In the inhomogeneous system, the right hand side becomes $\Psi_{kl\alpha \beta}$ for the corresponding values of  $k,l,\alpha, \beta$.

Collecting all terms of the kind $z^k\zz^l\bar z^0\bar\zz^0u^j$ gives: 
\begin{equation}ih_{kl}(u)=0,\,\,k,l\geq 0, \label{e1} \end{equation} 
except for $(k,l) = (0,0), (1,0), (2,0)$.
For these values we obtain, respectively,
\begin{equation} i h_{00 } -  i \bar h_{00 } = 0,\label{e2} \end{equation}
\begin{equation} 2 \bar f_{00 } + i  h_{10 } = 0,\label{e3} \end{equation}
\begin{equation} \bar g_{00 } +  i  h_{20 } = 0. \label{e4} \end{equation}
For terms of the form $z^k\zz^l\bar z^1\bar\zz^0u^j,\,$ we get for $k > 0$ and $ l\geq 0$: 
\begin{equation}-h'_{k-1,l}+2f_{kl}=0,\,\,k,l\geq 0,\label{e5} \end{equation}
except when  $(k,l) = (1,0),  (2,0),  (3,0),  (1,1).$ 
%In this case 
If $k = 0$ and $l >1$ we obtain 
\begin{equation} 2 f_{0l} = 0.\label{e6} \end{equation}
For   $z \bar z \z $ terms, i.e. $(k,l) = (0,1)$ we obtain
\begin{equation}  2 \bar g_{00} + 2 f_{11} - h'_{01} = 0.\label{e7} \end{equation}
Further, for $z \bar z $ terms, i.e. $(k,l) = (1,0)$,  we have
\begin{equation} - h'_{00} - \bar  h'_{00} + 2f_{10} + 2\bar f_{10} = 0.\label{e8} \end{equation} 
For $\z \bar z$ terms , i.e. $(k,l) = (0,1)$, we obtain
\begin{equation} 2 f_{01} + 2 \bar f_{00}= 0,\label{e9} \end{equation}
and for $z^2 \bar z $ terms , i.e. $(k,l) = (2,0)$,
\begin{equation} - h'_{10} + 2 f_{20} - 2i  \bar f'_{00} + \bar g_{10}= 0.\label{e10} \end{equation}
Finally for $z^3 \bar z$ terms, i.e. $(k,l) = (3,0)$, we obtain
\begin{equation}  -i  \bar g'_{00} + 2 f_{30} - h'_{20} = 0.\label{e11} \end{equation}

Next, consider terms of the form $z^k\zz^l\bar z^2\bar\zz^0u^j$. For $ k > 1$ and $l>0$, except for $(k,l) = (2,1)$,  we obtain 
\begin{equation}- \frac12 h'_{k, l-1} - \frac{i}{2}  h''_{k-2, l} + 2i  f'_{k-1, l} + g_{kl} = 0.   \label{e12} \end{equation}
When $k > 4$ and $l=0$, i.e. for the coefficients of $z^k \bar z^2$ we have
\begin{equation}  - \frac{i}{2}  h''_{k-2, 0} + 2i  f'_{k-1, 0} + g_{k0} = 0.  \label{e13} \end{equation} 
If  $k = 1$ and $l>0$ i.e. for the coefficients of $z \z^l  \bar z^2$, except for $(k,l) = (1,1)$,  we have 
\begin{equation}  - \frac{1}{2}  h'_{1, l-1} + 2i  f'_{0, l} + g_{1l} = 0.  \label{e14} \end{equation} 
Finally, if $k = 0$ and $l>2$, i.e. for $\z^l \bar z^2  $ we obtain 
\begin{equation}  - \frac{1}{2}  h'_{0, l-1}  + g_{0l} = 0.  \label{e15} \end{equation} 
Further, for $(k,l) = (0,2)$ we obtain  
\begin{equation}\bar g_{00}- \frac12 h'_{0, 1}  + g_{02} = 0.   \label{e16} \end{equation}
For $(k,l) = (1,1)$ we obtain  
\begin{equation} -3i \bar f'_{00} + 2\bar g_{10}  - \frac{1}{2}  h'_{10} + 2i  f'_{01} + g_{11} = 0.   \label{e17} \end{equation}
For $(k,l) = (3,0)$ we obtain  
\begin{equation}-i  \bar g'_{10} - \frac{i}{2}  h''_{10} + 2i  f'_{20} - \bar f_{00}'+ g_{30} = 0.   \label{e18} \end{equation}
For $(k,l) = (4,0)$ we obtain  
\begin{equation} - \frac{i}{2}  h''_{20} + 2i  f'_{30} - \bar g''_{00} + g_{40} = 0.   \label{e18A} \end{equation}
For $(k,l) = (0,1)$ we obtain
\begin{equation}- \frac{1}{2} h'_{00} - \frac{1}{2}  \bar h'_{00} +  2 \bar f_{10} + g_{01} = 0.   \label{e19} \end{equation}
For $(k,l) = (2,0)$ we obtain
\begin{equation} \frac{i}2 \bar h''_{00} - \frac{i}{2}  h''_{00} + 2i  f'_{10} + g_{20} -2i \bar f'_{10}  + \bar g_{20} = 0.\label{e20} \end{equation}
For $(k,l) = (2,1)$  we obtain
\begin{equation} -\frac52 i \bar g'_{00} - \frac{1}{2}  h'_{20} -  \frac{i}2  h'_{00}  + 2i f'_{11} + g_{21} = 0.   \label{e21} \end{equation}
% Further, for terms  of the form $z \bar z \z \bar\z u^j$ we have: 
% \begin{equation}\Re (- h'_{00 }+2f_{10} + 2 g_{10})  =0 \label{} \end{equation}
Collecting all terms of the kind $z^3 \bar\z u^j$ gives: 
\begin{equation}- \frac12 h'_{10 }+2f_{20} - i  \bar f'_{00}  =0. \label{e22} \end{equation}
Collecting all terms of the kind $z^3 \bar z^3  u^j$ gives: 
\begin{equation}\Re (\frac16  h'''_{00 } - f''_{10} + i g'_{20})  =0. \label{e23} \end{equation}
Further, collecting terms with $z^4\bar\zeta u^j$ we get:
\begin{equation}-\frac{1}{2}h'_{20}+2f_{30}-\frac{i}{2}\bar g_{00}'=0.\label{e24} \end{equation}
Collecting all terms of the kind $z^3   \bar z \bar \z     u^j$ gives: 
\begin{equation}- \frac{i}{2}   h''_{00 }+3i f'_{10}  +  2  g_{20}  + \frac{i}{2}  \bar h''_{00} -  i  \bar f'_{10} -i \bar g'_{01}  =0. \label{e25} \end{equation}
Finally, collecting all terms of the kind $z^4  \bar z \bar \z  u^j$ gives: 
\begin{equation}-\frac{i}{2}h_{10}''+3if_{20}'-\bar f_{00}''+2g_{30}-\frac{i}{2}\bar g_{10}'=0.\label{e26} \end{equation}

Now all the terms appearing in the normal form space conditions \eqref{Nspace}, \eqref{Nspace2}, \eqref{Nspace3} are considered, and we have to show that the resulting system of equations for $f_{kl},g_{kl},h_{kl},\,k,l\geq 0$ determines the latter ones uniquely. We do it step-by-step. We also make use of the following 
\smallskip

\noindent{\bf Convention.} In what follows, dots stand for linear expressions in the previously determined   coefficient functions $f_{kl},g_{kl},h_{kl}$ and the given right hand side functions $\Psi_{kl\alpha \beta}$.

\smallskip

\medskip

Equation \eqref{e1} determines  $h_{kl}$ for all $(k,l)$, except for  $h_{00}$,  $h_{10}$  $h_{20}$.
Equation \eqref{e2} determines  $\Im h_{00}$.
Equation \eqref{e3} expresses   $h_{10}$ by $f_{00}$, namely
\begin{equation}
 h_{10} = 2i \bar f_{00} + \dots.
\end{equation}
Equation \eqref{e4} expresses   $h_{20}$ by $g_{00}$, namely
\begin{equation}
 h_{20} = i \bar g_{00} + \dots.
\end{equation}
Equation \eqref{e5} determines   $f_{kl}$ for $ k > 0$,  except for  $f_{10}$,  $f_{20}$  $f_{30}$ and $f_{11}$.
Equation \eqref{e6} determines   $f_{0l}$,   except for  $f_{00}$,  $f_{01}$.
Equation \eqref{e7} expresses   $f_{11}$ by $g_{00}$, namely
\begin{equation}
 f_{11} = -  \bar g_{00} + \dots.
\end{equation}
Equation \eqref{e8} expresses   $\Re f_{10}$ by $\Re h'_{00}$, namely
\begin{equation}
\Re f_{10} = \frac12 \Re h'_{00} + \dots.
\end{equation}
Equation \eqref{e9} expresses   $f_{01}$ by $f_{00}$, namely
\begin{equation}
 f_{01} = -  \bar f_{00}  + \dots.
\end{equation}
Equation \eqref{e10} expresses   $f_{20}$ by $f'_{00}$ and $g_{10}$, namely
\begin{equation}
 f_{20} = 2i \bar f'_{00} - \frac12 \bar g_{10} + \dots.
\end{equation}
Equation \eqref{e11} expresses   $f_{30}$ by $g'_{00}$, using \eqref{e4}, namely
\begin{equation}
 f_{30} = \frac12 h'_{20} + \frac{i}2 \bar g'_{00} + \dots= i \bar g'_{00} + \dots.
\end{equation}
Equation \eqref{e12} determines  $g_{kl}$ for $k>1$ and $l>0$,  except for  $g_{21}$.
Equation \eqref{e13} determines  $g_{k0}$ for $k>3$. In particular,  $g_{40}$
is expressed by  $g'_{00}$, as
\begin{equation}
 g_{40} = \frac{i}2 h''_{20} - 2i   f'_{30} +  \bar g''_{00} + \dots = \frac52 \bar g''_{00} + \dots.
\end{equation}
Equation \eqref{e14} determines  $g_{1l}$ for $l>0$, except for  $g_{11}$. 
Equation \eqref{e15} determines  $g_{0l}$, except for  $g_{00}$, $g_{01}$ and $g_{02}$.
Equation \eqref{e16} expresses   $g_{02}$ by $g_{00}$, namely
\begin{equation}
g_{02} = - \bar g_{00}+\dots.
\end{equation}
 Equation \eqref{e17} expresses   $g_{11}$ by $g_{10}$ and $f_{00}$, namely
\begin{equation}
g_{11} = 6i \bar f'_{00} - 2\bar g_{10}+\dots.
\end{equation}
Equation \eqref{e18} expresses   $g_{30}$ by $g'_{10}$ and $f'_{00}$, namely
\begin{equation}
g_{30} = 2i \bar g'_{10} +4 \bar f''_{00}+ \dots.
\end{equation}
Equation \eqref{e19} gives  $g_{01}$ by $\Im f_{10}$ and $\Re h_{00}$, namely
\begin{equation}
\Im g_{01} = 2 \Im f_{10}  + \dots.
\end{equation}
and 
\begin{equation}
\Re g_{01} = \Re h'_{00} - 2 \Re f_{10} + \dots.
\end{equation}
Equation \eqref{e20} gives  $\Re g_{20}$ by $\Im f_{10}$,  namely
\begin{equation}
\Re g_{20} = 2 \Im f'_{10}+\dots.
\end{equation}
Equation \eqref{e21} gives  $g_{21}$ by $g'_{00}$ and $\Re h_{00}$,  namely
\begin{equation}
g_{21} = \frac52 i \bar g'_{00} + \frac{i}2 \Re h'_{00}+ \dots.
\end{equation}
Next, equations \eqref{e22} and \eqref{e26} give a system for $f_{00}$ and $g_{10}$, namely 
\begin{equation}
2i \bar f'_{00} - \bar g_{10} +\dots= 0
\end{equation}
and 
\begin{equation}
 6 \bar f''_{00} - \frac{i}{2} \bar g'_{10} +\dots= 0.
\end{equation}
Equation \eqref{e24} determines $\bar g'_{00}$, by
\begin{equation}
i \bar g'_{00} + \dots = 0.
\end{equation}
The real part of equation \eqref{e25} determines $\Im f'_{10}$, 
\begin{equation}
-2 \Im f'_{10} + \dots = 0.
\end{equation}
The imaginary part of \eqref{e25} together with \eqref{e23} give a system of two real equations for
$\Re h_{00}$ and $\Im g_{20}$, namely 
\begin{equation}
 \Re h''_{00} +2 \Im g_{20}+\dots= 0
\end{equation}
and 
\begin{equation}
\frac16 \Re h'''_{00} - \Im g'_{20}+\dots  = 0,
\end{equation}
which determines $\Re h_{00}$ and $\Im g_{20}$. 
\end{proof}

\begin{remark}
It was pointed out to us by Joel Merker that the formula \eqref{e18A} in an earlier version of this paper contained a missing term (which doesn't eventually effect our calculations).
\end{remark}

\autoref{directsum} and the relations \eqref{homolog} imply, in the standard manner, the following proposition, which is the first part of \autoref{main}.
\begin{proposition}\Label{formalnf}
For the germ at a point $p$ of any (formal of real-analytic) everywhere $2$-nondegenerate hypersurface in $\CC{3}$, there exists a formal transformation $H:\,(\CC{3},p)\lr(\CC{3},0)$ mapping $M$ into a hypersurface in normal form \eqref{nspace}. In any (formal or holomorphic) local coordinates \eqref{germ}, a normalizing transformation is unique up to the right action of the $5$-dimensional  stability group $G$ of the model \eqref{cone}.
\end{proposition}

\section{Convergence of the formal normal form and applications}

The proof of \autoref{main} is accomplished by \autoref{formalnf} and the following convergence theorem.

\begin{theorem}\Label{converge}
Any formal transformation bringing a uniformly $2$-nodnegenerate real-analytic hypersurface \eqref{germ} to the normal form \eqref{Nspace} -- \eqref{Nspace3} is convergent.
\end{theorem}

\begin{proof}
We will show that the transformation \eqref{normalmap} bringing a hypersurface $M$, as in \eqref{germ}, to a normal form is convergent. Recall that the factorization \eqref{factord} allows to deal with all the other normalizing transformations.

Our proof of convergence relies on a series of Chern-Moser type simplifications of the defining function, and the construction of the so-called {\em chain field}, integration of which  defines certain canonical curves (chains)  in a 2-nondegenerate hypersurface. For Chern-Moser type modifications we mostly provide key aspects of the procedure  and leave some of the details to the reader (since the modifications that we use are analogous to that by Chern-Moser), while the choice of a chain (and subsequently sending it into the canonical curve \eqref{Gamma}) is a more subtle step in the normalization procedure and we address it in detail. The proof is split into several steps. For each of the Chern-Moser type simplifications, it turns out that the previously achieved normalization conditions are preserved, {\em except} a transformation which is a part of the very last Step VIII below. The latter is discussed in detail in the respective step. 
 
 We keep the notations

\begin{equation}\Label{def}
v=\Phi(z,\z,\bar z,\bar\z,u)=\sum_{k,l,\alpha,\beta\geq 0}\Phi_{kl\alpha\beta}(u)z^k\zeta^l\bar z^\alpha\bar\zeta^\beta=P(z,\zeta,\bar z,\bar\zeta)+\sum_{k,l,\alpha,\beta\geq 0}\Psi_{kl\alpha\beta}(u)z^k\zeta^l\bar z^\alpha\bar\zeta^\beta
\end{equation}
for the defining equation of $M$.

\subsection{Proof of Theorem 4}

\mbox{}

\medskip

\noindent{\bf Step I:  choice of a transverse curve.} We make a special choice of a smooth real-analytic curve $\gamma$ transverse at $0$ to the complex tangent. The choice is specified later in Step VIII.

\smallskip

\noindent{\bf Step II: removing pluriharmonic terms.} Next, we do a local biholomorphism at $0$ which eliminates the pluriharmonic terms in \eqref{def} (that is, we get $\Psi_{kl00}=0$) and, at the same time, straightens the curve $\gamma$, that is, $\gamma$ becomes 
\begin{equation}\Label{Gamma}
\Gamma=\{z=\zeta=0,\,\,\im w=0\}.
\end{equation}
The latter is possible due to e.g. \cite{ber},\cite{lmblowups}. In what follows, we consider only transformation preserving \eqref{Gamma}.    

\smallskip

\noindent{\bf Step III: cleaning the $2$-jet.} We perform a gauge transformation $$z\mapsto f(w)z,\quad \zeta\mapsto\zeta,\quad w\mapsto w$$ with an appropriate $f(w),\,f(0)=0,\,f'(0)=1$ in order to achieve $\Psi_{1010}(u)=0$ (we use $\Psi_{1010}(0)=0$). Such a transformation preserves the earlier normalization conditions. Let us provide details for this, very first, Chern-Moser type simplification. Consider the basic identity 
\begin{equation}\label{Basic}
\im h(z,\z,w)=\left.\Phi^*\bigl(f(z,\z,w),g(z,\z,w),\overline{f(z,\z,w)},\overline{g(z,\z,w)},\re h(z,\z,w)\bigr)\right|_{w=u+i\Phi(z,\z,\bar z,\bar\z,u)}
\end{equation}
($\Phi,\Phi^*$ are the source and the target defining functions, respectively).
Let $A(u)=\Phi_{1010}(u)$.  Then we set: $f(u):= \sqrt {A(u)}$. Then we pick the transformation: $$z \mapsto zf(w),\, w \mapsto w.$$ The curve \eqref{Gamma} is clearly preserved under such transformation. Let us then plug to the basic identity $\bar z=\bar\z=0$. Due to the absence of pluriharmonic terms in the initial defining function, we (for  $\bar z=\bar\z=0$ ) have $w=u$. Now the basic identity gives $0=\Phi^*(zf(u),0,\z,0,u)$, which gives  $\Phi^*(z,0,\z,0,u)=0$ so that the normalization achieved in Step II is preserved. Finally, to show that $\Psi^*_{1010}(u)=0$, we compare in the basic identity terms with $z\bar z u^l$, and get, by using the absence of pluriharmonic terms: $A(u)=f^2(u) (1+\Psi^*_{1010}(u))$, which implies the desired property in view of our choice of $A(u)$.

Further, we use the nondegeneracy of the term $z\bar z$ in \eqref{def} to eliminate terms $z\bar\zeta u^m$ in \eqref{def} by means of a transformation $$z\mapsto z+f(w)\zeta,\quad\zeta\mapsto\zeta,\quad w\mapsto w$$ with an appropriate $f(w)$. Such an $f(w)$ can be easily found if taking the basic identity \eqref{Basic} and picking in it terms $z\bar\z u^l$, which gives $\Psi_{1001}(u)=f(u)+\Psi^*_{1001}(u)$. We make sure that the earlier conditions are preserved, in a similar manner to the previous simplification. Geometrically, the transformation means straightening the Levi kernels along $\Gamma$ (alternatively, one can consider the $3$-dimensional variety constructed as the union of Levi curves through points of $\Gamma$ and then a holomorphic transformation, straightening the surface and the curves and preserving $\Gamma$). We end up with the additional normalization condition $\Psi_{1001}=0$ achieved. Furthermore, since the Levi rank along $\Gamma$ is constantly $1$, this also implies $\Psi_{0101}=0$.  Note that now the {\em only} term of degree $\leq 2$ in $z,\bar z,\zeta,\bar \zeta$  in $\Phi$ is $z\bar z$.

\smallskip

\noindent{\bf Step IV: cleaning the $3$-jet.} We first perform a gauge transformation
$$z\mapsto z,\quad \zeta\mapsto g(w)\zeta,\quad w\mapsto w$$ 
with an appropriate $g(w),\,g(0)=0,\,g'(0)=1$ in order to achieve $\Psi_{2001}(u)=0$ (we use $\Psi_{2001}(0)=0$). This simplification is very analogous to that in Step III. Next, 
we perform a transformation 
$$z\mapsto z+z^2f(w),\quad \zeta\mapsto \zeta+g(w)z, \quad w\mapsto w$$
with appropriate $f(w),g(w)$ to eliminate in \eqref{def} the term $\Psi_{2010}$ and at the same time $\Psi_{3001}$ (which is actually a part of the {\em $4$-jet} in $z,\bar z,\zeta,\bar \zeta$). To accomplish this, we first check
%(analogously to the above) 
that the previously achieved normalization conditions are preserved. For pluriharmonic terms, we use the same argument as above. Further, this transformation  affects neither two jets nor $\Psi_{2001}$, hence all previously achieved normalizations are preserved.  Next, we collect in the basic identity \eqref{Basic} terms with $z^2\bar zu^l$ and $z^3\bar\z u^l$ respectively and obtain the following transformation rules:
$$ \Psi_{2010} \mapsto \Psi_{2010}+f(u)+\frac{1}{2}\bar g(u)+\cdots, \quad \Psi_{3001}\mapsto \Psi_{3001}+f(u)+\cdots,$$
where dots stand for  expressions analytic in $u$ and polynomial in $f,\bar f,g,\bar g$, which  either have degree $\geq 2$ in $f,\bar f,g,\bar g$, or have degree $1$ in the above but then have a factor vanishing at $u=0$. The latter follows from the previously achieved normalization conditions (compare also with \eqref{e10}, \eqref{e22}). Now the desired choice of $f,g$ is accomplished by applying the implicit function theorem. 

We next perform a transformation of the kind $$z \mapsto z+f(w)\z^2, \quad \z\mapsto\z,\quad w\mapsto w.$$
By comparing the respective terms in the basic identity, we see, first of all, that all the previous normalization conditions are preserved, and that the transformation rule
$$ \Psi_{0210} \mapsto \Psi_{0210}+f(u)$$
holds. Now the normalization condition $\Psi_{0210}=0$ is accomplished by an appropriate choice of $f(w)$.

Finally, we remove in this step the $z\zeta\bar z$ term by a transformation   
$$z\mapsto z+z\zeta f(w),\quad \zeta\mapsto \zeta, \quad w\mapsto w$$ with an appropriate $f(w)$ (analogously to a similar transformation in Step III). We then get $\Psi_{1110}=0$. To make sure the previously achieved normalization conditions are preserved, we compare in the basic identity \eqref{Basic} the respective terms (and use the fact that the $2$-jet in $z,\z,\bar z,\bar\z$ is ``cleaned'' according to Step III). 

Note that, arguing identically to the proof of \autoref{goodperturb} (more precisely, using identities of the kind \eqref{iterative} arising from the uniform Levi-degeneracy), it is easy to conclude that the terms $\zeta^2\bar\zeta u^0$ and $z\zeta\bar\zeta u^0$ are not present now in \eqref{def}.  At the same time, we observe that one can shift the basic point $(0,0,0)$ to $(0,0,u_0),\,u_0\in\RR{}$ (the latter is still within the curve $\Gamma$, as in \eqref{Gamma}). Then the shifted function $\Psi$ must satisfy the same property, so that we get $\Psi_{0201}(u)=\Psi_{1011}(u)=0$. 

We end up with a hypersurface \eqref{def} for which {\em all}  terms of degree $\leq 3$ in $z,\bar z,\zeta,\bar \zeta$  in $\Psi$ vanish.  Further, repeating the above argument of using identities of the kind \eqref{iterative}, we get for the fourth order terms  $\Psi_{\bold k}$: 
\begin{equation}\Label{4ord}
(\Psi_{\bold k})_{\zeta\bar\zeta}=0, \quad |k|=4.
\end{equation}

\smallskip

\noindent{\bf Step V: choosing an orthonormal basis within the Levi kernel.} We now perform a transformation
$$z\mapsto ze^{i\varphi(w)},\quad \zeta\mapsto \zeta e^{2i\varphi(w)}+g(w)z^2,\quad w\mapsto w,$$
where $\varphi(w),g(w)$ satisfy $\varphi(0)=0,\,\varphi(\RR{})\subset\RR{},\,g(\RR{})\subset\RR{},$ in order to achieve simultaneously $\re\Psi_{3011}=0$ and $\Psi_{2020}=0$. Geometrically, the latter transformation fixes an ortonormal basis within the Levi kernel $\{\im w=0,\,z=0\}$ along the transverse curve $\Gamma$.  Indeed, for a hypersurface satisfying the previous normalization conditions, we consider the basic identity \eqref{Basic}, compare the $z^3\bar z \bar\z u^l,z^2\bar z^2 u^l$ terms and compute  (using, in particular, \eqref{4ord}) that the desired terms change as:
$$\re\Psi_{3011}(u)\mapsto \re\Psi_{3011}(u)-\frac{1}{2}\varphi'(u), \quad \Psi_{2020}(u)\mapsto \Psi_{2020}(u)-2\varphi'(u)+g(u).$$
This shows the desired choice of $\varphi,g$. We make sure, analogously to the above, that the previously achieved normalizations are preserved.

\smallskip

\noindent{\bf Step VI: removing terms\, $hol\cdot\bar z$\, and\, $hol\cdot\bar z^2$.}  For a hypersurface \eqref{def} satisfying the normalization conditions achieved in Steps I-V, we perform a Chern-Moser type  transformation of the kind 
\begin{equation}\Label{killk1}
z\mapsto z+f(z,\zeta,w),\,\,\zeta\mapsto\zeta,\,\,w\mapsto w,
\end{equation} 
where $f$ preserves the origin and has degree at least $3$ in $z,\zeta$.
If one denotes now the sum of all nenzero terms in $\Psi$ of the kind $z^k\zeta^l\bar z^1\bar\zeta^0 u^m$ with  $k+l\geq 3$ by $\chi(z,\zeta,u)\bar z$, then,  using the previous normalization conditions, we see that one can simply take $$f(z,\zeta,w):=\chi(z,\zeta,w)$$  
and obtain a hypersurface satisfying, in addition, $\Psi_{kl10}=0$. This is by collecting in the basic identity \eqref{Basic} the terms under consideration. Notably, the previously achieved normalizations {\em are} preserved: indeed, the $2$-jet in $z,\z,\bar z,\bar\z$ is ``cleaned'' according to Step III, and the function $f$ (which has in $z,\z$ only the linear term and terms of degree $\geq 3$), that is why $f$ contributes only to terms of degree $\geq 4$ in  $z,\z,\bar z,\bar\z$, and in degree $4$ only to the ones of the kind $hol\cdot\bar z$. This shows that the previous normalizations are preserved. 

Further, we perform a Chern-Moser type  transformation of the kind 
\begin{equation}\Label{killk2}
z\mapsto z,\,\,\zeta\mapsto\zeta+g(z,\zeta,w),\,\,w\mapsto w
\end{equation}
where $g$ preserves the origin and has degree at least $2$ in $z,\zeta$.
If one denotes then the sum of all nonzero terms in $\Psi$ of the kind $z^k\zeta^l\bar z^2\bar\zeta^0 u^m$ with  $k+l\geq 2$ by $\rho(z,\zeta,u)\bar z^2$, then a direct calculation shows that one can simply take $$g(z,\zeta,w):=\rho(z,\zeta,w)$$  
and obtain a hypersurface satisfying also $\Psi_{kl20}=0$ (we simply compare in the basic identity the respective terms, as above). We make sure, analogously to the previous simplification, that the other normalizations are preserved. This implies, in particular,
\begin{proposition}\Label{killed+}
For an analytic everywhere $2$-nondegenerate hypersurface, a transformation eliminating the terms \eqref{killed}, as suggested by  \cite{bk}, can be chosen to be holomorphic.
\end{proposition}

We now apply \autoref{goodperturb} and conclude that the hypersurface $M$, in particular, satisfies the condition \eqref{perturb}.  
Furthermore, collecting all the above information, we conclude that {\em all  terms of degree $\leq 4$ in $z,\bar z,\zeta,\bar \zeta$  in $\Psi$ vanish. }

\smallskip

We shall now discuss in a separate section the distinguished choice of a transverse curve mentioned in Step I.

\subsection{Chains in an everywhere $2$-nondegenerate hypersurface}

We are now aiming to specify a curve chosen in Step I. An appropriate choice of this curve leads to the additional normalization conditions $\Psi_{4001}=\Psi_{4011}=0$. A curve with this property is called {\em a chain}. The chains in our constructions arise as (the projections of) integral curves of an appropriate analytic direction field $d$ defined on a certain bundle over $M$. The direction field $d$, even though being {\em analytic}, is constructed by using a {\em formal normal form}. The idea of such an approach to the definition of chains and hence to the proof of convergence of a normal form is originally due to Zaitsev \cite{zaitsevnf} who applied it to revisiting the convergence problem for the classical Chern-Moser normal form. We address $d$ as {\em the chain field of an everywhere $2$-nondegenerate hypersurface}. The chain field  was also applied by Kossovskiy and Zaitsev \cite{generic,cmhyper} for proving the convergence of normal forms for finite type hypersurfaces. Here we develop the approach of Zaitsev in the uniformly Levi-degenerate case. 

The necessity to use a {\em bundle} over $M$ for constructing the chains is, first of all, due to the presence of the kernel of the Levi form, so that one has to consider the quotient $TM/K$ (where $K$ is the distribution of Levi kernels). Second, it is due to the ``large'' uncertainty in the choice of normal form arising from the action \eqref{factored} of the stability group of the light cone \eqref{cone} on the normal forms. Particularly the action of the component $\g_1$ of the automorphism algebra is essential here, as this algebra acts transitively on directions {\em in the quotient $TM/K$} transverse to $T^{\CC{}}M/K$.  This results in considering a bundle over $M$, fibers of which can be essentially identified with the above $\g_1$ component (the bundle $X$ below and its open dense subset $X^{\mathrm{o}}$).

Our precise construction is as follows.

Recall that the complex tangent bundle $T^{\CC{}}M$ is endowed with the canonical (holomorphically invariant) subbundle $K$, where $K_p$ is the Levi kernel for $M$ at $p$. Hence, we may consider the canonical (holomorphically invariant) bundle $TM / K$ over $M$ as well as its projectivization 
\begin{equation}\Label{pbundle}
X:=\mathbb P (TM/K).
\end{equation} 
$X$ can be interpreted as the projective bundle of directions $l$ in the fibers of $TM/K$. We will be also interested in the open subset $X^{\mathrm{o}}\subset X$ corresponding to the directions in the fibers of $TM/K$ transverse to the fibers of the subbundle $T^{\CC{}}M/K$. Let us fix the notation $M,M^*$ for the  source and the target manifolds, respectively.
%It naturally contains the complex subbundle $T^{\CC{}}M/K$. 
Then every biholomorphism $H:\,M\mapsto M^*$ naturally extends to one between $TM / K$ and  $TM^* / K^*$ and hence to a fiber-preserving biholomorphism of the projectivizations \eqref{pbundle} (and the respective open subsets $X^{\mathrm{o}}\subset X$ and $ (X^*)^{\mathrm{o}}\subset X^*$). Further, any smooth (unparameterized) curve $\gamma\subset M$ naturally lifts to a curve in the projective bundle $X$ (indeed, every parameterized curve naturally lifts to $TM$ by considering its tangent vector at a point,  hence it lifts naturally to $TM/K$, and so an unparameterized curve lifts in this way to the projectivization of $TM/K$).   In turn, if $M$ is in the form \eqref{germ},  the fiber of $TM/K$ at $0\in M$ can be identified with the real subspace in $\CC{3}$ given by $$\{z=z_1, w=u_1,\zeta=0\}_{z_1\in \CC{}, u_1\in \RR{}},$$ while $T_0M$ can be identified with the real subspace in $\CC{3}$ given by $$\{z=z_1, w=u_1,\zeta=\zeta_1\}_{z_1,\zeta_1\in \CC{}, u_1\in \RR{}}.$$ A direction in $T_0M/K_0$ transverse to $T^{\CC{}}_0M/K_0$ corresponds then to a (real) line $l = \{z_1 = au_1\}, a \in \CC{}$, i.e., we have the coordinate $a$ for the fiber of $X$ at $0$. The variable $a$ has then the meaning of the partial 1-jet $\frac{dz}{du}$ of a curve $z=z_1(u_1),w=u_1,\zeta=\zeta(u_1)$ at the point $0$ and we write $a_1$ for the partial 2-jet $\frac{d^2z}{d^2u}$ at the point $(0,a)$ for the coordinate of the tangent space to the fiber of $X$ at $0$. The collection 
\begin{equation}\label{tangcoord}
(z_1, \zeta_1, u_1, a_1)
\end{equation} 
as above can be considered as a tuple of coordinates for the tangent space to $X$ at a point lying in the fiber of $X$ at $0$. 

The bundles $TM / K$ and $X$ are very convenient for describing  transformations into a normal form $M^*$ geometrically. We consider now {\em biholomorphic} transformations to the normal form, while the formal case is discussed 
in \autoref{formalok} below. Namely, for $p\in M$ and for each direction $l_p\subset T_pM/K_p$ transverse to $T^{\CC{}}_pM/K_p$ (such a direction, once again, is seen as an element of the fiber of $X$ at $p$ lying also in $X^{\mathrm{o}}$), there is a normalizing transformation $H:\,(M,p)\mapsto (M^*,0)$ such that its extension to $X$ maps $(p,l_p)$ into $(0,L_0)$, where $L_0$ is the point in $X^*$ corresponding to the "vertical direction"  \eqref{Gamma}. To see this, let us first switch to coordinates \eqref{germ} (using a specific polynomial transformation, as described below). We can then restrict here to transformations which have the {\em identity differential} on the complex tangent $T^{\CC{}}_pM$. 
 Then a choice of $l_0$ corresponds to the action of the component $\mathfrak g_1$ of the stability algebra $\g$ of the light cone on the normal forms by the formula \eqref{factored}. Indeed, the linear part of a transformation from the flow  $\mathfrak g_1$ acts on the local coordinates as $z\mapsto z+aw,\,u\mapsto u$ in $T_0/K_0M$; in this way, any direction  in $T_0/K_0M$ corresponding  to a  line $l=\{z=au\},\,a\in\CC{}$ can be mapped into the vertical one. 
  Next, the subalgebras $\mathfrak g_0^c$ and $\mathfrak g_2$ from the isotropy algebra \eqref{isotropy} preserve the curve $\Gamma$. (The condition of having the identity differential on the complex tangent means that  the subalgebra $\mathfrak g_0^c$ does not act here in fact). Now \eqref{factord} implies the existence of a desired transformation. The above argument also shows that a choice of a direction $l_0$ is in correspondence with the complex number $a=f_w(0)$, where $H=(f,g,h)$ is a normalizing transformation with the identity differential on the complex tangent. 
  
 We further observe that the subalgebras $\mathfrak g_0^c$ and $\mathfrak g_2$ in fact preserve the condition of being in normal form (unlike the subalgebra $\g_1$!). To see this, for a vector field $X=f\dz+g\frac{\partial}{\partial\z}+h\dw$ from the span of the subalgebras $\mathfrak g_0^c$ and $\mathfrak g_2$ and a hypersurface $v=\Phi(z,\z,\bar z,\bar\z,u)$ in normal form, we write down the tangency condition 
 $$\im h=\re\bigl(2f\Phi_z+2g\Phi_{\z}+h\Phi_u\bigr),$$  
and see that the latter identity holds {\em modulo terms in the normal form space $\mathcal N$}, as in \eqref{nspace}. Treating now the flow of $X$ by exponentiating (e.g., \cite{ilyashenko}), we see that the flow of     $X$ maps the source hypersurface into a one in normal form, as required.
 
 We conclude from here that any two transformations $H_1,H_2$ bringing $M$ to a normal form and mapping the same direction $l_p$ into $L_0$ are related as 
\begin{equation}\Label{reltd}
H_2=\psi\circ H_1
\end{equation} 
($\psi$ is a transformation generated by the flows of $\mathfrak g_0^c$ and $\mathfrak g_2$).

Let us now fix $p\in M$ and a pair $(p,\bold v)\in X^{\mathrm{o}}$, and consider a  transformation $H:\,(M,p)\mapsto (M^*,0)$ bringing $M$ into a  normal form $M^*$ at $p$, and such that the induced map $\tilde H:\,X\mapsto X^*$ maps $(p,\bold v)$ into $(0,L_0)$. Let $\tilde\Gamma$ be the lifting of the curve $\Gamma$, as in \eqref{Gamma}, to the bundle $X^*$. Consider finally
\begin{equation}\Label{field}
d\tilde H^{-1}|_{(0,L_0)}(T_0\tilde\Gamma)\subset T_{(p,\bold v)}X.
\end{equation}
For a fixed $\tilde H$, \eqref{field} defines a direction in the tangent space $T_{(p,\bold v)}X$. We claim that the latter direction does {\em not} depend on the choice of $H$. Indeed, this follows from \eqref{reltd} and the fact that $\mathfrak g_0^c$ and $\mathfrak g_2$  preserve $\Gamma$, hence their extensions to $X$ preserve $T_0\tilde\Gamma$ (because both the natural extensions of a map and the curve are unique, as discussed above). 

The latter means that \eqref{field} defines a direction field in $X^{\mathrm{o}}$. 
\begin{definition}
We call the direction field $d$, as in \eqref{field}, {\em the chain field} of a uniformly $2$-nondegenerate hypersurface. 
\end{definition}
In turn, in any coordinates \eqref{germ}, we may consider normalizing transformations $H=(f,g,h)$ such that their extensions map $(0,\bold v)\in X^{\mathrm o}$ into $(0,L_0)$, and further restrict to transformations with the identity differential on the complex tangent $T^{\CC{}}_0M$. (The latter means the conditions $f_z(0)=g_\z(0)=h_w(0)=1$.) With such an arrangement, the value $d$ of the chain field at $(0,\bold v)$ is in correspondence with the complex numbers $f_{ww}(0),g_w(0)$.    (And, as discussed above, a choice of a direction $\bold v$ is in correspondence  the complex number $a=f_w(0)$). Indeed, choose coordinates for the tangent space of  $X$ at the reference point from the fiber of $0$ as in \eqref{tangcoord}. Now, when determining the direction field \eqref{field} at the reference point, we are concerned with the action of the differential of the inverse (prolonged) normalizing transformation on $T_0\widetilde{\Gamma}$ (which is spanned in local coordinates by $(0,1,0,0)$). With the above choice of local coordinates, the resulting direction is hence determined by $F_w(0),H_w(0),G_w(0),F_{ww}(0)$ (where $(F,G,H)$ are the components of the inverse transformation). We easily verify that $$H_w(0)=1/h_w(0)=1,\,G_w(0)=-g_w(0),\,F_w(0)=-f_w(0)=-a,\,F_{ww}(0)=-f_{ww}(0).$$ In this way, $f_{ww}(0),g_w(0)$ become {\em functions of $a$}. We will later show that these functions are in fact real-analytic. Furtheremore, in case the direction $\bold v$ is ``vertical'' (i.e. coincides with $L_0$), and the chain field $d$ is "vertical" too at the reference point (i.e. spanned by $T_0\widetilde{\Gamma}$), we get with our arrangements: $f_{ww}(0)=g_w(0)=0$ (and also $f_w(0)=a=0$). We will significantly use this fact in what follows. 

\begin{remark}\label{formalok} It is immediate from the construction of the chain vector field that, for each fixed $p\in M$, its value at any point $(p,\bold{v})$ in the fiber of $X$ at $p$ {\em depends only on the $2$-jet of an appropriate normalizing transformation}. The latter is well defined for merely {\em formal} transformations. In view of this, the definition of the chain field can be word-by-word extended to the situation when all the normalizing transformations are merely formal.   
\end{remark} 

 We next prove
\begin{proposition}\Label{analfield}
The direction field in $X^{\mathrm{o}}$ defined by \eqref{field} is in fact analytic.
\end{proposition} 
\begin{proof}
First, we note that the coordinates \eqref{germ} (actually achieved by a bi-cubic change of variables in $\CC{3}$) depend on the point $p$ analytically. This can be easily seen from the procedure in \cite{ebenfeltC3} and the implicit function theorem.

Second, as follows from \autoref{killed+}, terms \eqref{killed} can be removed by a holomorphic transformation (and hence the representation \eqref{perturb} can be achieved by the same holomorphic transformation). In turn, we may choose such transformation as a composition of Steps I to  VIII above (with an arbitrary choice of a transverse curve in Step I). The explicit procedure in Steps I to VIII combined with the implicit function theorem imply then that coordinates \eqref{perturb} depend on the point $p$ analytically as well.  

The latter means that, for proving the proposition, we have to prove the following claim: {\em for a hypersurface \eqref{perturb} and a (formal) normalizing transformation $H=(f,g,h)$ with $f_w(0)=a$, the functions $\chi(a):=f_{ww}(0),\,\tau(a):=g_w(0)$ are real-analytic in $a$}.  

The proof of the claim is a slight modification of the proof of a similar claim in e.g. \cite{generic}.  We note that, for a normalizing transformation $H=(f,g,h)$, the parameter $a$ is a part of the collection $H_j=(f_{j-1},g_{j-2},h_j)$ with $j=3$, while $\chi(a)$ is a part of the one with $j=5$ and $\tau(a)$ is a part of the one with $j=4$. The collections $H_4,H_5$ are obtained by solving the equations \eqref{homolog}, with the initial data corresponding to the element of the flow of $\mathfrak{g}_1$ in \eqref{isotropy}   with $f_w(0)=a$. In view of \eqref{homolog} and \autoref{directsum} (arguing by induction),  all the collections $H_j$ with $j\geq 3$ are obtained by solving a  system of linear equations  with nondegenerate (and fixed) matrix and the right hand side depending polynomially on $a,\bar a$. That is why, in particular, $H_4,H_5$ are polynomial in $a,\bar a$ and so are $\chi(a),\tau(a)$, as required. This proves the claim and the proposition. 
\end{proof}

%The precise choice of a curve $\gamma$ being transformed in Step I into \eqref{Gamma} is accomplished by the following lemma.
%\begin{lemma}\Label{choice}
%The curve $\gamma$ is step one can be chosen in such a way that, after the procedure in Steps I -- VII, one has $\Psi_{4001}=\Psi_{4011}=0$.
%\end{lemma}
%\begin{proof}[Proof of \autoref{choice}]
%The proof is similar, in many respects, to the proofs of analogous statements in e.g. \cite{chern}, \cite{generic}. 
%\end{proof}

We now integrate the direction field \eqref{field} and obtain a foliation of $X^{\mathrm o}$ by smooth real-analytic (unparameterized) curves $\tilde\gamma$.
%,each of which is pointwise transverse to $T_{(p,\bold v)}T^{\CC{}}M/K.$ 
This leads to the following
\begin{definition}\Label{degchain}
Canonical projections of the curves $\tilde\gamma$ from $X^{\mathrm o}$ to $M$ are called {\em chains}. 
\end{definition} 
As follows from the above procedure, through each point $p$ there is a unique chain in a fixed direction $l_p\subset T_pM/K_p$ transverse to $T^{\CC{}}_pM/K_p$. Furthermore, importantly, {\em the family of chains is biholomorphically invariant}, as follows from its definition. As follows, again, from the definition, {\em chains are mapped by normalizing transformations into the standard "vertical"\, curve \eqref{Gamma}}. 

We shall remark that, as discussed in the Introduction, orbits of the linear part of the stabilizer \eqref{isotropy} of the model do {\em not} act transitively on transverse directions anymore (unlike the Levi-nondegenerate situation). It is not difficult to compute that, for a hypersurface \eqref{germ}, the orbit of the "vertical"\, direction \eqref{Gamma} in $T_0M$ is the cone
\begin{equation}\Label{cancone}
\left\{(z,\z,u):\,\,\zeta u+iz^2=0,\,\,u\neq 0\right\}\subset T_0M.
\end{equation}
\begin{definition}
The cone \eqref{cancone} is called the {\em canonical cone} for $M$ at $0$.
\end{definition}
Clearly, the canonical cone at $p$ does not depend on the choice of coordinates \eqref{germ} and is furthermore biholomorphically invariant. That is, each everywhere $2$-nondegenerate hypersurface $M$  is equipped with a field of canonical cones in tangent spaces. 
Possible directions of chains form in the tangent space are {\em precisely} the directions from the canonical cone. 

\subsection{End of proof of Theorem 4} 

\mbox{}

\medskip

\noindent{\bf Step VII: applying the chain property.} We  are now able to specify Step I below. Namely, for a hypersurface \eqref{germ}, we choose $\gamma$ to be the unique chain in the "vertical" direction (i.e. the direction corresponding to the line \eqref{Gamma}). 

We finally have to prove that, with the above choice of $\gamma$, we have 
$\Psi_{4001}=\Psi_{4011}=0$ upon completion of Steps I -- VI. In view of the invariancy in $u$ of the pre-normal form achieved in Steps I to VI, it is enough to  prove $\Psi_{4001}(0)=\Psi_{4011}(0)=0$.  Consider a (formal) transformation $H=(f,g,h)$, as in \eqref{normalmap}, bringing $(M,0)$ into a normal form. By the definition of the chain, it satisfies: 
\begin{equation}\Label{ura}
f_{ww}(0)=0,\quad g_w(0)=0.
\end{equation} 
In view of that and the outcome of Step VI, when solving the equations \eqref{homolog} for the map $H$, for $j=3$ we get the same result as for the identity map. For $j=4$, it is straightforward to see from the basic identity  \eqref{basic}, the outcome of Step VIII and the fact that $H$ coincides with the identity map to weight $3$ that the weight $4$ identity gives {\em precisely} the same equations for the collection $H_4=(f_3,g_2,h_4)$ as in the formal procedure in \autoref{directsum}, besides the only equation for   the $z^4\bar\zeta$ terms (analogous to \eqref{e24}) which gives $\Psi_{4001}(0)$ in the right hand side instead of $0$. Then, by using the second condition in \eqref{ura}, we conclude that $\Psi_{4001}(0)=0$. Very similarly, when considering $j=5$ and the $z^4\bar z\bar\zeta$ terms, we use the first condition in \eqref{ura} and obtain $\Psi_{4011}(0)=0$.

As a result, we end up with a hypersurface satisfying, in addition, $\Psi_{4001}=\Psi_{4011}=0$.

\smallskip

\noindent{\bf Step VIII: choice of a parameterization along the chain.} The last remaining conditions for $\Psi$ to belong to the normal form space are $\im\Psi_{3011}=0$ and  $\Psi_{3030}=0$. 

 We achieve the desired conditions in two sub-steps. We first perform a transformation
\begin{equation}\label{last}
z\mapsto f(w)z,\,\,\zeta\mapsto\zeta+ig(w)z^2,\,\,w\mapsto h(w),
\end{equation}
where $$f(0)=1,\,\,h(0)=h''(0)=0,\,\,h'(w)=f^2(w),\,\,f(\RR{})\subset\RR{},\,\,g(\RR{})\subset\RR{}.$$

We then compare the respective terms in the basic identity \eqref{Basic} and compute (employing also the previously achieved normalization conditions) that: 
\begin{equation}\Label{comptd}
\Psi_{3011}\mapsto \im\Psi_{3011}h'-\frac{1}{2}h''+\frac{1}{2}gh',\quad \Psi_{3030}\mapsto h'\Psi_{3030}+\left(\frac{1}{6}h'''-\frac{1}{2}\frac{h''^2}{h'}\right)-\frac{1}{2}g'(u).
\end{equation}
Now the conditions $\im\Psi_{3011}=\Psi_{3030}=0$ turn \eqref{comptd} into a system of analytic nonsingular ODEs, which we solve uniquely with the Cauchy data $h(0)=0,\,h'(0)=1,\,h''(0)=0$. (To see this, one has to solve the first equation for $g$ and substitute the result into the second, which makes  the second equation a nonsingular third order ODE in $h$; the latter is solved uniquely with the above initial data, and then $g$ is found from the substitution). 

We now have to take care of the previously achieved normalization conditions. By comparing all the normal form terms in the basic identity \eqref{Basic}, we see that {\em almost} all of them are preserved, with the sole exception of terms of the conditions $\Psi_{jl20}$ with $j+l\geq 5$ (i.e. nonzero terms $z^j\z^l\bar z^2u^m$ with $j+l\geq 5$  can appear after the transformation \eqref{last}; they "arrive" from terms of the kind $z^j\z^l\bar\z u^m$ with $j+l\geq 5$). The restriction $j+l\geq 5$ comes from the identity \eqref{4ord} and the analogous identity with $|k|=5$ applicable for a hypersurface obtained after Step VII.   We thus have to perform a "final cleaning transformation"
$$z\mapsto z, \quad \z \mapsto \z +g(z,\z,w)z^2, \quad w \mapsto w,$$
where $g$ is an appropriate function {\em with vanishing $4$-jet in $z,\z$ at the origin}. Applying now the basic identity \eqref{Basic} and considering all the normal form terms, we see that only terms of total degree $\geq 7$ in $z,\z,\bar z,\bar\z$ are touched by such transformation,  while pluriharmonic terms and terms $hol\cdot\bar z$  still vanish.  Very similarly to Step VI, we see that an appropriate choice of the function $g$ makes all the terms $z^j\z^l\bar z^2u^m$ in the new defining function $\Psi$ vanish, as desired. Now {\em all} the normal form conditions are finally achieved.      

We remark that an important part of Step VIII is a choice of a parameterization along a chain, very analogously to the situation of Chern-Moser.

This completely proves the theorem. 

\end{proof}

\autoref{converge} immediately implies \autoref{main}.

\subsection{Proofs of further results.} 
As mentioned above, \autoref{main} allows for  important applications stated in \autoref{main2} and \autoref{main3}. We give the proofs for these theorems below. 

\begin{proof}[Proof of \autoref{main2}]
According to \cite{pocchiola}, the local equivalence of an everywhere $2$-nondegenerate hypersurface to the model \eqref{cone} amounts to vanishing of the two basic invariants $W$ and $J$ (given by lengthy expressions in terms of the defining function and the CR-vector fields, which we do not provide here). It is straightforward to check then that, for a hypersurface in normal form \eqref{nspace}, the invariant $W|_0$ is proportional to the normal form coefficient $\Psi_{3002}(0)$, and, as long as   $\Psi_{3002}(0)$ vanishes, the invariant $J|_0$ is proportional to the normal form coefficient $\Psi_{5001}(0)$. This immediately implies the assertion of the theorem.
\end{proof}

\begin{proof}[Proof of \autoref{main3}]
For a hypersurface \eqref{model+} in normal form \eqref{nspace}, we argue as in the proof of \autoref{goodperturb} and consider its complex defining equation $w=\theta(z,\z,\bar z,\bar\z,\bar w)$. The uniform $2$-nondgeneracy of $M$ then gives the determinant equation 
\begin{equation}\label{deteqn}
\begin{vmatrix}
\theta_{\bar z} & \theta_{\bar \zz} & \theta_{\bar w} \\
\theta_{z\bar z} & \theta_{z\bar \zz} & \theta_{z\bar w} \\
\theta_{\zz\bar z} & \theta_{\zz\bar \zz} & \theta_{\zz\bar w} 
\end{vmatrix}=0.
\end{equation}
We need now to express the latter equation as a PDE for the function $\Phi$, as in \eqref{model+}. For doing so, we have to establish relations between the $2$-jets of the real defining function $\varphi(z,\z,\bar z,\bar\z,u):=P(z,\z,\bar z,\bar\z)+\Phi(z,\z,\bar z,\bar\z,u)$ and the complex defining function $\theta$. Considering the identity 
$$\frac{1}{2i}(\theta-\bar w)=\varphi\left(z,\zeta,\bar z,\bar\z,\frac{1}{2}(\theta+\bar w)\right)$$
as an identity in $(z,\zeta,\bar z,\bar\z,\bar w)$, we differentiate the latter once in each of the variables and easily conclude:
\begin{equation}\label{1jet}
\theta_{\bar z}=\frac{2i\varphi_{\bar z}}{1-i\varphi_u},\,\,\theta_{\bar \z}=\frac{2i\varphi_{\bar \z}}{1-i\varphi_u},\,\,\theta_{\bar w}=\frac{1}{1-i\varphi_u},\,\,
\end{equation}
Differentiating then once more, we get:
\begin{equation}\label{2jet}
\begin{aligned}
&\theta_{z\bar z}=2i\varphi_{z\bar z}+\cdots,\,\,\theta_{z\bar \z}=2i\varphi_{z\bar \z}+\cdots,\,\,\theta_{z\bar w}=i\varphi_{zu}+\cdots,\\
&\theta_{\z\bar z}=2i\varphi_{\z\bar z}+\cdots,\,\,\theta_{\z\bar \z}=2i\varphi_{\z\bar \z}+\cdots,\,\,\theta_{\z\bar w}=i\varphi_{\z u}+\cdots,
\end{aligned}
\end{equation}
where dots stand for terms of degree $\geq 2$ in the first and second order derivatives of $\varphi$. It is not difficult to see then, by using \eqref{cone}, that the relations \eqref{1jet},\eqref{2jet} turn \eqref{deteqn} into a PDE of the kind:
\begin{equation}\label{thePDE}
\Phi_{\z\bar\z}=T\bigl(z,\bar z,\z,\bar\z,\Phi_{\bar z},\Phi_{\bar\z},\Phi_u,\Phi_{z\bar z},\Phi_{z\bar \z},\Phi_{zu},\Phi_{\z\bar z},\Phi_{\z u}\bigr),
\end{equation}
where $T$ is a {\em universal} (i.e. independent on $M$) rational function in all its variables with $T(0)=0,\,dT(0)=0$ (the function $T$ can be deduced explicitly, even though the resulting expression is cumbersome and we do not provide it here). The PDE \eqref{thePDE} with the initial data on the ``cross'' $\z\bar\z=0$ is a particular case of the classical {\em Goursat problem}. We now make use of the result of Lednev (see \cite{lednev},\cite{wagschal},\cite{wagschal2}) suggesting, in particular, that a PDE of the kind  \eqref{thePDE} with analytic at the origin right hand side satisfying $T(0)=0,\,dT(0)=0$ has a unique analytic at the origin solution with $\Phi(z,\zeta,\bar z,0,u)=\Phi(z,0,\bar z,\bar\z,u)=0$. (Note that, as long as $\Phi$ satisfies the reality condition, the second condition automatically follows from the first one). To apply the result of Lednev for \eqref{thePDE} with any analytic at the origin initial data 
\begin{equation}\label{data}
\Phi(z,\zeta,\bar z,0,u)=:\chi(z,\z,\bar z,u)\in\mathcal D,
\end{equation}   
we first shift $\Phi$ by substructing from it the associated real-analytic expression \eqref{disting}, and switch to the modified PDE. We now note that, since functions in $\mathcal D$ all have vanishing $4$-jet at the origin, the modified PDE (which shall be already considered as one with the zero initial data, as required in Lednev's theorem) also has vanishing linear part at the origin. Now Lednev's theorem implies that the PDE \eqref{thePDE} has a unique analytic near the origin solution for any analytic at the origin initial data \eqref{data}. To see that the solution $\Phi$ belongs to the normal space \eqref{nspace}, we note that all possible monomials in \eqref{nspace} vanishing of which determines the normal form space all have either $\z$ or $\bar\z$ absent. But monomials of the latter kind which are present in $\Phi$ (or their conjugated) must belong to  $\mathcal D$ due to the initial data \eqref{data}, while monomials from $\mathcal D$ satisfy the normal form conditions by definition of $\mathcal D$. Thus $\Phi$ belongs to the normal form space.

 Finally, it is not difficult to see that the PDE \eqref{thePDE} is real, that is, it is a real-analytic (rational) PDE for a scalar real function $\Phi$ in the variables $(x,y,s,t,u),\,z=x+iy,\,\z=s+it$ (it can be seen from its invariance under the involution $z\leftrightarrow\bar z,\,\z\leftrightarrow\bar \z$). The reality condition in the space $\mathcal D$ means that the initial data under consideration amounts to a real-analytic initial data on the ``cross''\, $st=0$ which is compatible on the intersection $s=t=0$. The latter means that the solution satisfies the reality condition $\Phi(z,\z,\bar z,\bar\z,u)\in\RR{}$. This immediately implies the assertion of the theorem.

\end{proof}


\begin{thebibliography}{20000000}

\bibitem[BER99]{ber}  M. S. Baouendi, P. Ebenfelt, L. P. Rothschild, 
``Real Submanifolds in Complex Space and Their Mappings''.
Princeton University  Press, Princeton Math. Ser. {\bf 47},
Princeton, NJ, 1999.
\bibitem[BER96]{beralg} M. S. Baouendi,  P. Ebenfelt, L. P. Rothschild,  {\it Algebraicity
of holomorphic mappings between real algebraic sets in $C^n$}. Acta
Math. 177 (1996), no. 2, 225--273.
\bibitem[BHR96]{bhr} S. Baouendi, X. Huang and L.P. Rothschild,  {\it Regularity of CR
mappings between algebraic hypersurfaces.} Invent. Math. 125
(1996), 13--36.
\bibitem[Bel05]{belC3} V. K. Beloshapka, {\it Symmetries of real hypersurfaces of a three-dimensional complex space.} (Russian) Mat. Zametki 78 (2005), no. 2, 171--179.
\bibitem[Bel79]{belold} V. K. Beloshapka,  {\it The dimension of the group of automorphisms of an analytic hypersurface} (Russian)
Izv. Akad. Nauk SSSR Ser. Mat. 43 (1979), no. 2, 243-–266.
\bibitem[BK15]{bk} V. K. Beloshapka, I. Kossovskiy, {\it The sphere in $\CC{2}$ as a model surface for degenerate hypersurfaces in $\CC{3}$}. Russ. J. Math. Phys. 22 (2015), no. 4, 437–-443.
\bibitem[Ca32]{cartan} \'E. Cartan,  {\it Sur la g\'eom\'etrie pseudo-conforme des hypersurfaces de l'espace de deux
variables complexes II}. Ann. Scuola Norm. Sup. Pisa Cl. Sci. (2)
1  (1932), no. 4, 333--354.


%\bibitem[Ca24]{cartan2} \'E. Cartan. {\it Sur les vari\'et\'es \`a connexion projective.} Bull.
%Soc. Math. France {\bf 52} (1924), 205-241.
\bibitem[Ca84]{catlin} D. Catlin, {\it Boundary invariants of pseudoconvex domains.} Ann. of Math. (2) 120 (1984), no. 3, 529–-586.

\bibitem[CM74]{chern} S. S. Chern and J. K. Moser,  {\it Real hypersurfaces in
complex manifolds},  Acta Math.  {133}  (1974), 219--271.
\bibitem[Eb98]{ebenfeltC3} P.~Ebenfelt,  {\it Normal forms and biholomorphic equivalence of real hypersurfaces in $\mathbb C^3$}. Indiana Univ. Math. J. 47 (1998), no. 2, 311-366.
\bibitem[Eb06]{ebenfeltduke} P. Ebenfelt,   {\it Uniformly Levi degenerate CR manifolds: the 5-dimensional case}. Duke Math. J. 110 (2001), no. 1, 37-80. Correction in Duke Math. J. 131 (2006), no. 3, 589--591.
\bibitem[ES96]{es} V.~Ezhov and  G.~Schmalz,  {\it Normal form and two-dimensional chains of an elliptic CR-manifold in $\mathbb C^4$}. J. Geom. Anal. 6 (1996), no. 4, 495--529.

\bibitem[FK08]{kaup} G.~Fels, W.~Kaup,  {\it Classification of Levi degenerate homogeneous CR-manifolds in dimension 5.} Acta Math. 201 (2008), no. 1, 1-–82.

\bibitem[FK07]{kaup2} G.~Fels, W.~Kaup,   {\it CR-manifolds of dimension 5: a Lie algebra approach.} J. Reine Angew. Math. 604 (2007), 47–-71.
\bibitem[Fr77]{freeman} M. Freeman,  {\it Local biholomorphic straightening of real submanifolds}, Ann. of Math. (2) 106 (1977), no. 2, 319--352

\bibitem[HY09]{hy} X.~Huang and W.~Yin,  {\it A Bishop surface with a vanishing Bishop invariant.} Invent. Math. 176 (2009), no. 3, 461--520.

\bibitem[IY08]{ilyashenko} Y.~Ilyashenko and S.~Yakovenko. Lectures on
analytic differential equations. Graduate Studies in Mathematics,
86. American Mathematical Society, Providence, RI, 2008.

\bibitem[IZ13]{iz} A. Isaev and D. Zaitsev,  {\it   Reduction of five-dimensional uniformly Levi degenerate CR structures to absolute parallelisms}. J. Geom. Anal. 23 (2013), no. 3, 1571--1605.

\bibitem[KZ06]{kaupzaitsev}  W. Kaup, D. Zaitsev, {\it On local CR-transformation of Levi-degenerate group orbits in compact Hermitian symmetric spaces.} J. Eur. Math. Soc. (JEMS) 8 (2006), no. 3, 465–-490.

\bibitem[K05]{kol05} M. Kol\'a\v r, {\it Normal forms for hypersurfaces of finite type in $ \mathbb C^2$}, Math. Res. Lett., {12} (2005),  897--910.

\bibitem[KMZ14]{kmz} M.~Kol\'a\v r, F. Meylan, D. Zaitsev,   {\it
 Chern-Moser operators and
polynomial models
in CR geometry}, Advances in  Mathematics 263 (2014), 321-–356.
\bibitem[KM18]{km} M. Kol\'a\v r, F. Meylan,
  \textit{Nonlinear CR automorphisms of Levi degenerate
hypersurfaces  and a new gap phenomenon},
Annali della Scuola Normale Superiore di Pisa, 19 (2019), 847-868.

%\bibitem[KKZ16]{kkz} M. Kol\'a\v r, I. Kossovskiy, D. Zaitsev, {\it Normal forms in Cauchy-Riemann Geometry: a survey}. Analysis and geometry in several complex variables, 153-–177, Contemp. Math., 681, Amer. Math. Soc., Providence, RI, 2017.

\bibitem[KLX16]{klx} I.\,Kossovskiy. B.\,Lamel and M.\,Xiao,  {\it Regularity of CR-mappings into Levi-degenerate hypersurfaces}. Comm. in Analysis and Geometry, 29 (2021), 151-181.

\bibitem[KZ19]{generic} I.~Kossovskiy and D.~Zaitsev, {\it Convergent normal form for real hypersurfaces at generic Levi degeneracy.} J. Reine Angew. Math. (Crelle's Journal) 749 (2019), 201-–225. 

\bibitem[KZ15]{cmhyper} I.~Kossovskiy and D.~Zaitsev, {\it Convergent normal form and canonical connection for hypersurfaces of finite type in $\mathbb C^2$}. Advances in Mathematics, 2015, pp. 670--705.

\bibitem[KrLo83]{krlo} N. Kruzhilin, A. Loboda,   
{\it Linearization of local automorphisms of pseudoconvex surfaces.} (Russian) 
Dokl. Akad. Nauk SSSR 271 (1983), no. 2, 280–-282. 

\bibitem[LM07]{lmblowups} B.~Lamel and N.~Mir, {\em Finite jet determination of local CR automorphisms through resolution of degeneracies.} Asian J. Math. 11 (2007), no. 2, 201–-216.

\bibitem[LS19]{ls} B.\,Lamel, L.\,Stolovitch. {\it Convergence of the Chern–Moser–Beloshapka normal forms}. Journal fur die reine und angewandte Mathematik (Crelle's Journal), 765 (2020), 205-247.

\bibitem[Le48]{lednev} N. A. Lednev, {\it A new method for the solution of partial differential equations.} (Russian)
Mat. Sbornik N. S. 22(64), (1948). 205-–266.


\bibitem[MS14]{ms} C. Medori , A. Spiro, {\it The equivalence problem for 5-dimensional Levi degenerate CR manifolds,} Int. Math. Res. Not. IMRN 2014, no. 20, 5602–-5647.

 \bibitem[MP19]{pocchiola2}
 J. Merker, S. Pocchiola, 
  {\it Explicit absolute parallelism for 2-nondegenerate real hypersurfaces $M^5 \subseteq \CC{3}$ of constant Levi rank 1},
 Journal of Geometric Analysis, 30 (2020),2689--2730. 

\bibitem[Mir17]{mirdef} N. Mir, {\it Holomorphic deformations of real-analytic CR maps and analytic regularity of CR mappings.} J. Geom. Anal. 27 (2017), no. 3, 1920–-1939.

\bibitem[Poc13]{pocchiola} S. Pocchiola, {\it Explicit absolute parallelism for 2-nondegenerate real hypersurfaces in $\CC{3}$ of constant Levi rank 1.} Preprint, arXiv:1312.6400 (2013).




\bibitem[Po07]{poincare} H.~Poincar\'e, {\it Les fonctions analytiques de deux variables et la
representation conforme}. Rend. Circ. Mat. Palermo. (1907) {
23}, 185--220.

\bibitem[PZ17]{zelenko} C. Porter, I. Zelenko, {\it Absolute parallelism for 2-nondegenerate CR structures via bigraded Tanaka prolongation,} Journal fur die reine und angewandte Mathematik (Crelle's Journal) 777 (2021), 195-250.

\bibitem[Sta96]{stanton2} N.~Stanton, {\it Infinitesimal CR automorphisms of real hypersurfaces.} Amer. J. Math. 118 (1996), no. 1, 209-–233.

\bibitem[Ta62]{tanaka} N. Tanaka, {\it On the pseudo-conformal geometry of hypersurfaces of the space of n complex variables.}
J. Math. Soc. Japan {14} 1962 397--429.

\bibitem[Wa78]{wagschal2} C. Wagschal. Le probl\`eme de Goursat non-lin´eaire. S´eminaire Equations aux d´eriv´ees partielles (Polytechnique), (1978-1979), p. 1-11, exp. 17. url =
http://www.numdam.org/item/SEDP 1978-1979 A17 0. 

\bibitem[Wa79]{wagschal} C. Wagschal,  {\it
Le probl\`eme de Goursat non lin´eaire.} (French)
J. Math. Pures Appl. (9) 58 (1979), no. 3, 309-–337.


\bibitem[XY17]{xiao} M.\,Xiao, Y.\,Yuan, {\it Complexity of holomorphic maps from the complex unit ball to classical domains}. Asian J. Math., 
22(2018), 729-760.

\bibitem[Za14]{zaitsevnf} D.\,Zaitsev. Private communication, 2014.









 \end{thebibliography}
\end{document}